\documentclass[12pt, reqno]{amsart}

\usepackage{graphicx}

\usepackage{amssymb,amsmath,amsthm}
\usepackage{esint}
\usepackage{color}
\usepackage[left=1.5cm, right=1.5cm, top=2cm, bottom=2cm]{geometry}

\theoremstyle{definition}

\newtheorem{theorem}[equation]{Theorem}
\newtheorem{lemma}[equation]{Lemma}
\newtheorem{corollary}[equation]{Corollary}

\newtheorem{definition}[equation]{Definition}

\newtheorem{remark}[equation]{Remark}
\newtheorem{proposition}[equation]{Proposition}

\usepackage{graphicx}

\newcommand{\norm}[1]{{\left \Vert{#1}\right \Vert}}

\usepackage{amsmath,amsthm,amssymb,enumerate, fancyhdr}
\usepackage[english]{babel}
\usepackage[backref=page]{hyperref}

\usepackage{thmtools}

\newcommand{\setR}{\mathbb{R}}
\newcommand{\N}{\mathbb{N}}
\newcommand{\divergence}{\mathrm{div}}
\newtheorem*{theorem*}{Theorem}

\newcommand{\bfq}{\mathbf{q}}

\newcommand{\Fcal}{\mathcal{F}}

\DeclareMathOperator{\supp}{supp}
\DeclareMathOperator{\cof}{cof}

\newcommand{\abs}[1]{|#1|}

\newcommand{\absB}[1]{\Bigl|#1\Bigr|}
\newcommand{\absBB}[1]{\biggl|#1\biggr|}

\newcommand{\toweakstar}{\overset{*}\rightharpoonup}

\providecommand{\skp}[1]{\langle{#1}\rangle}
\newcommand{\nablasym}{{\nabla_s}}

\DeclareMathOperator{\divg}{div}

\DeclareMathOperator{\test}{\Fcal_{\eta}}
\DeclareMathOperator{\testn}{\Fcal_{\eta_{n}}}
\DeclareMathOperator{\testd}{\Fcal_{\delta}}

\newcommand{\bn}{{\nu}}
\newcommand{\bfp}{\mathbf{p}}

\newcommand{\R}{\mathbb{R}}

\newcommand{\bu}{\mathbf{u}}

\newcommand{\by}{{y}}

\numberwithin{equation}{section}

\usepackage{graphicx}
\usepackage{tikz}

\begin{document}

\title[Time-periodic weak solutions for FSI]{Time-periodic weak solutions for the interaction of an incompressible fluid with a linear Koiter type shell under dynamic pressure boundary conditions}
%Newtonian fluid interacting with a Koiter type shell under no-slip boundary conditions.}

%    Information for first author
\author{Claudiu M\^{i}ndril\u{a}}
%    Address of record for the research reported here
\address{Department of Analysis, Faculty of Mathematics and Physics, Charles University,
Sokolovsk\'{a} 83, 18675, Prague, Czech Republic }
\email{mindrila@karlin.mff.cuni.cz}
%    \thanks will become a 1st page footnote.
%\thanks{The first author was supported in part by NSF Grant \#000000.}

%    Information for second author
\author{Sebastian Schwarzacher}
\address{Department of Mathematics, Uppsala University, Box 480
    751 06 Uppsala \& Department of Analysis, Faculty of Mathematics and Physics, Charles University,
Sokolovsk\'{a} 83, 18675, Prague, Czech Republic }
\email{schwarz@karlin.mff.cuni.cz}
%\thanks{Support information for the second author.}

%    General info
%\subjclass[2000]{Primary 54C40, 14E20; Secondary 46E25, 20C20}

\date{\today}

%\dedicatory{This paper is dedicated to our advisors.}

%\keywords{Mathematics!}

\begin{abstract}
In many occurrences of fluid-structure interaction time-periodic motions are observed. We consider the interaction between a fluid driven by the three dimensional Navier-Stokes equation and a two dimensional linearized elastic Koiter shell situated at the boundary. The fluid-domain is a part of the solution and as such it is changing in time periodically. On a steady part of the boundary we allow for the physically relevant case of dynamic pressure boundary values, relevant to model inflow/outflow. We provide the existence of at least one weak time-periodic solution for given periodic external forces that are not too large. For that we introduce new approximation techniques and a-priori estimates.
\end{abstract}

%\noindent\textsc{MSC (2010):35Q35 (primary); 35Q30, 35J60, 35J70, 35J75.} 

\maketitle

\section{Introduction}

Periodic motions are often observed when fluids and solids are interacting. This phenomenon ranges from large motions as the rotation of wind-wheels to flutter. It includes important applications such as heart-beat driven blood flow through vessels or oscillations of bridge decks~\cite{bonheure2019periodic, pironneau1994optimal,Quarteroni2000,FSIforBIO,C21} and important benchmarks from numerics as in~\cite{turek.s.hron.j:numerical}. In this work some answers to the related relevant question {\em under what conditions time-periodic motions may occur mathematically} are given.

The interaction between fluids and solids is an intensive domain of research. There is a vast list of results in many scientific areas. While most results are in the applied sciences including applied mathematics more and more pure mathematical results are being  proved. For an overview on the analytical results and related applications for the set-up considered here, that is the setting of a moving shell constituting a part of the boundary of the fluid with which it is in interaction, we refer to the excellent survey \cite{C21} and the references therein.
As can be seen there, the focus in the study of {\em weak solutions} is on the {\em existence of solutions for the Cauchy problem}, that is solutions whose initial values are given. We mention here the pioneering results obtained by Grandmont et. al.~\cite{Gr05,Gr08} for the case of a plate, and by Lengeler and R\r{u}{\v{z}}i{\v{c}}ka in \cite{LR14} for the case of a linear shell. Further extensions were obtained in \cite{MS22}  where the Cauchy problem for the \emph{nonlinear Koiter energy} was proved, along some additional regularity estimates. For the initial value problem further extension are compressible or heat-conducting fluids
\cite{BreSch18,breit-schwarz-fourier} and the {\em constructive} existence results by means of Arbitrary-Lagrangian-Eulerian methods obtained by Muha and Cani{\'c} in \cite{muha-canic-arma-13,muha-canic-noslip}. There dynamic pressure conditions where investigated that are prominent for blood-vessel simulations.  See also \cite{KamSchSpe23} where deformations in all coordinate directions were considered.

In comparison to the Cauchy problem, much less research was devoted to time-periodic problems. Even so time-periodic solutions are very natural to appear and are related to some of the most relevant phenomena, like  turbulence or bifurcation, see \cite{yang2020periodically,galdi2016bifurcating}.

One of the first results in this regime was obtained by the authors of this paper~\cite{our-paper}. Please see there other references for previous results on time-periodic solutions for strong solutions, fixed fluid-domain or most significantly rigid body motions~\cite{galdi2006existence,galdi2020viscous}.
Very recently, Kreml et. al. obtained in \cite{kreml2023time}  the existence of time-periodic weak solutions for a 2D compressible fluid which interacts with an elastic beam. In this work we continue the study on {\em time-periodic solutions for fluid-structure interactions} and extend to to the regime of {\em elastic shells}. For the fluid we consider the {\em unsteady three dimensional Navier-Stokes equation} while the shell is moving in a prescribed direction along a potentially {\em curved} reference geometry. We consider a linearized Koiter energy type, rigorously justified in \cite{Ci05}. Hence the evolution of the solid is {\em hyperbolic and linear}. The shell is situated on a part of the fluid-boundary, which is changing in time accordingly. On a steady part of the fluid-boundary we consider pressure boundary values. Our main result is Theorem~\ref{thm:main}, where the existence of coupled time-periodic weak solutions to such interactions (namely~\eqref{eqn:system}) is shown for all time-periodic forces, whose size is not larger then a constant essentially depending on the curvature of the reference geometry in relation to the stiffness of the elastic shell. See Figure~\ref{figure} for a sketch of the geometric setup. 
The current work is influenced by the strategy developed by the authors in~\cite{our-paper}. The methodology developed here allows to extends the state of the art in two ways.
\begin{itemize}
\item The main emphasize of the present paper is to show the existence of time-periodic motions for {\em elastic shells}. This leads to the most dramatic differences in the analysis in contrast to flat elastic plates considered in~\cite{our-paper}. Indeed, in~\cite{our-paper} the motion of the shell was restricted to a fixed coordinate axis. In the present paper curved reference configurations are allowed and the motion of the shell is then along the non-trivial normal vector-bundle. Curved reference configurations are natural to be considered, for instance in case a blood vessel or a gas balloon is modeled. Considering curved reference conditions makes the analysis much more difficult as we explain below in Subsection~\ref{sec:math}. This can already be seen in case of the Cauchy problem, where it took almost ten years to generalize the theory from elastic plates to elastic shells~\cite{Gr05,LR14}.
Similar to the Cauchy problem in order to exclude {\em topological changes} further dependence of the principle curvatures of the reference configurations in relation to the size of the admissible forces need to be imposed. These dependencies are here however quantified in a controlled manner and vanish in the flat plate case.
\item In this paper we treat inflow/outflow boundary conditions, for which time-periodic solutions have not been shown previously; even so they allow to simulate blood-flow driven by the periodic force of a heart beat. What we show is that this quite prominent setting for applications (see ~\cite{C21}) seems to be very appropriate to expect the appearance of time-periodic motions. In our previous work~\cite{our-paper}, we considered everywhere no-slip boundary conditions. Every no-slip boundary conditions implies that the volume of the fluid is conserved which implies non-local motion effects of the shell. See Remark~\ref{rem:boundary} on further discussions on that issue.
\end{itemize}
The crucial technical reason for why there are much more existence results for the Cauchy problem in relation to time-periodic solutions lies in the fact that the so-called {\em energy estimates} are very different.
Indeed, as in the time periodic setting initial and end value cancel the {\em energy estimates} are here {\em significantly weaker}. In case the structure is not viscous, this means essentially no estimate for it by the {\em energy inequality} that is not coming from the fluid. Hence, a second {\em regularity estimate} for the structure is unavoidable. By this we mean an estimate for the elastic energy of the solid, see~\eqref{eqn:main-energy-estimate}. This was already realized in \cite{our-paper}. The technical innovation in the present paper however lies directly in the construction of the solution. Already the Galerkin basis (commonly the first step to find a solution) is constructed along the {\em extra regularity estimates} that are unavoidable in the time-periodic setting. Technically this means that the deformation itself is a test-function on the discrete level (see \eqref{eqn:sup-E-n-t}).
This construction allows to perform all necessary estimates in the very first approximation step.

We believe this innovation to be rather suitable for further applications. 
A few are summarized in several remarks that follow Theorem~\ref{thm:main}, where also limitations of the approach are discussed. This includes the treatment of further boundary conditions (see Remark~\ref{rem:boundary} and Remark~\ref{rem:noslip}).  For more information on the technical innovations see Subsection~\ref{sec:math}

\subsection{The model}
We introduce here the geometric setting, the elastic model, the fluid model and the coupled partial differential equations.
\subsubsection*{The geometric setting}
We define $\Omega\subset \mathbb{R}^3$ as the reference domain for the fluid, which is assumed to be open, bounded, connected  with  a $C^3$-boundary. We shall denote the tangential unit vectors by $\tau_1,\tau_2$ and the  outer normal by $\nu$. 
The boundary of $\Omega$ consists of three parts. First, $M$ the time-changing part of the boundary that is determined by the motion of the shell. We assume that this part is a single smooth connected component. The shell is assumed to be clamped at its endpoints. Hence this part of the boundary is naturally parametrized through the deformation of the structure. 
Second, $\Gamma_p$ the inflow/outflow part of the boundary. 
It is contained in the steady part of $\partial \Omega$. It is assumed to contain a finite number of smooth components. 
Finally, the reminder of the boundary is the Dirichlet part $\Gamma_D$ on which (for the sake of simplicity) we assume homogeneous no-slip boundary conditions $\mathbf{u}=\mathbf{0}$.\footnote{An inflow in terms of a non-zero Dirichlet boundary value would be meaningful. The present method seems to be suitable to allow for inhomogenuous periodic boundary values as well. Restrictions on the regularity or size of the boundary values should be expected.} 
Both Dirichlet and inflow/outflow boundary parts are steady in time, while the moving part of the boundary changes {\em via the solid deformation that is assumed to move in direction of the fixed direction of the outer normal of the domain}. In particular, this model reduction implies that the velocity in tangential direction of the fluid is assumed to be zero along the shell. This reduction is common in this setting,  see also~\cite{Ci05}. Indeed, the theory of weak-solution for shells that are modeled to move freely in all coordinate directions has only recently been initiated~\cite{KamSchSpe23}. 
The situation is depicted in the figure~\ref{figure}.
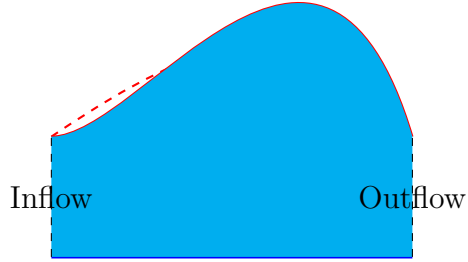
\begin{figure}\label{figure}
\caption{The domain $\Omega_{\eta(t)}$}
\begin{tikzpicture}[scale=0.8]
\draw[thick,red]   (8, 2) .. controls (9.5, 2) and (12.5, 7) .. (14, 2);
\draw[thick,dashed,red]   (8, 2) .. controls (9.5, 3) and (12.5, 5) .. (14, 2);
%\draw[thick,dashed]   (8.5, 0.5) .. controls (9.5, 1) and (12.5, 3) .. (13.5, 0.5);
%\draw[thick,dashed]   (7, 3.5) .. controls (9.5, 5) and (12.5, 8) .. (15, 3.5);
%\draw[thick,dashed]   (5.65, 0.2) .. controls (6.2,2.7) .. (7, 3.5);
%\draw[thick,dashed]   (14, 0)--(15,3.5);
%\draw[thick,dashed](8,0)--(14,0)--(14,2)--(8,2)--(8,0);
\draw[very thick, blue](8,0)--(14,0);
\path node at (11,1) {\textcolor{blue}{$\Omega_\eta$}};
\path node at (11,3.2) {\textcolor{red}{\bf $M$}};
\draw[->,thick, red] (12.2,3.47)--(12.4,4.2);
%\draw[<->,thick] (10.1,1.3)--(8.5,4.4);
\path node at (12.7,3.7) {\textcolor{red}{\bf $\eta\nu$}};
%\path node at (10,2.5) {\textcolor{black}{\bf $-L\nu$}};
%\path node at (9.3,3.7) {\textcolor{black}{\bf $L\nu$}};
\fill[color=cyan,opacity=.1]  (8, 2) .. controls (9.5, 2) and (12.5, 7) .. (14, 2) -- (14,0) -- (8,0) -- (8,2);
%\draw[thick,black]   (8, 1.2) .. controls (6.7, 1) and (4,0) .. (7,0);
%\draw[thick](7,0)--(12,0);
%\draw[thick,black]   (8, 1.2) .. controls (12.5, 3) .. (14, 1.2);
\draw[dashed](8, 0) -- (8, 2);
\draw[dashed](14, 0) -- (14, 2);
%\draw[thick,black]   (14, 1.2) .. controls (13.5, 0.5) .. (12, 0);
%\path node at (7,0.5) {\textcolor{black}{ supp$\psi^L$}};
\path node at (8,1) {\textcolor{black}{Inflow}};
\path node at (14,1) {\textcolor{black}{Outflow}};
%\path node at (12,1.3) {\textcolor{black}{ $Q_L$}};
%\path node at (11,5) {\textcolor{black}{ $Q^L$}};
\end{tikzpicture}
\end{figure}

Concerning the moving part $M$, we assume that it is parametrized by  a map $\phi$ and an open set $\omega \subset  \mathbb{R}^{2}$ via 
\begin{equation}
M=\phi\left(\omega\right),\ \phi\in C^{3}\left(\omega;\mathbb{R}^{3}\right),\ \phi\ \text{injective}
\end{equation}
This enables us to define the tangent vectors $\partial_{1} \phi (y)$ and $\partial_{2} \phi (y)$ at each point of $\phi(y)$, $y \in \omega$ and the unit normal vector field as 
\begin{equation}
\nu : \omega \mapsto S^{2}, \  \nu\left(y\right):=\frac{\partial_{1}\phi\left(y\right)\times\partial_{2}\phi\left(y\right)}{\left|\partial_{1}\phi\left(y\right)\times\partial_{2}\phi\left(y\right)\right|}, \quad y\in \omega.
\end{equation}

Let us now define, for $x$ in a neighbourhood of $\partial\Omega$, the quantities 
\[y\left(x\right):=\text{arg}\min_{y\in\omega}\left|x-\phi\left(y\right)\right|
\]
$s(x)$ which fulfills the relation
\begin{equation}
s\left(x\right)\nu\left(y\left(x\right)\right)+y\left(x\right)=x
\end{equation}

and 
\begin{equation}
\mathbf{p}\left(x\right):=\phi\left(y\left(x\right)\right).
\end{equation}
One easily notices that 
\[\left|s\left(x\right)\right|=\min_{y\in\omega}\left|x-\phi\left(y\right)\right|
\]
Let us consider, among the sets 
\[S_{l}:=\left\{ x\in\mathbb{R}^{3}:\text{dist}\left(x,\partial\Omega\right)<l\right\} ,\quad l>0
\]
where $y,  \ s, \ \mathbf{p}$ are defined, the one which is maximal and denote it by $S_{L_0}$. This means 
\[S_{L_0}=\left\{ x\in \R^3\,:\, s(x)\leq L_0,\right\}. 
\]
On the moving part we attach an \emph{elastic shell} which moves only in normal direction and 
 whose displacement field  is  $\eta \nu $.
In the following we fix $L<L_0$ and only consider deformations 
\[
\eta \in C^0_0( \omega,[-L,L]),
\]
which implies that no touching of the moving domain and hence no topological change of the moving domain is possible.
%Then $\eta$ is extended by zero to $\mathbb{R}^{2}$, which corresponds to a clamped shell.
The mapping
\begin{equation}\label{eqn:phi-eta-t}
\phi_{\eta}:\omega\mapsto\phi_{\eta}\left(\omega\right),\quad \phi_{\eta}:x\mapsto \phi\left(x\right)+\eta\left(t,x\right)\nu(\phi\left(x\right))
\end{equation} 
defines the parametrization of the moving part of the boundary $\partial\Omega_\eta$. 
In order to extend it to a parametrisation of the whole domain we introduce
\[
\psi_\eta :\partial\Omega \to \partial \Omega_\eta,\quad \bfp\mapsto\begin{cases}
    \phi_\eta\circ\phi^{-1}(\bfp)\text{ for }\bfp\in M
    \\
    \bfp\text{ otherwise.}
\end{cases}
\]
Observe that $\phi_{\eta}$ and $\psi_\eta$
are homeomorphism and even a $C^{k}$ diffeomorphism if $\eta \in C^{k}_0(\omega,[-L,L])$ and $\phi^{-1}\in C^{k}(M)$ for $k \in \mathbb{N}$.
% of the moving domain  $\Omega_\eta$,  which can be thought to be  the interior of the volume bounded by  $\partial \Omega_\eta$.
%Here $\chi_M$ denotes the characteristic function on $M$. 

Let $I=[0,T]$ be a time-interval. 
We now  introduce  the  \emph{deformed time-space cylinders} for times $t\in I$ and deformations $$\eta : I \times \omega \mapsto \mathbb{R}$$ 
 throughout \[
I\times\Omega_{\eta}:=\bigcup_{t\in I}\left\{ t\right\} \times\Omega_{\eta\left(t\right)}.
\]

The mapping $\psi_{\eta(t)}$ can be extended to a homeomorphism
\begin{equation}\label{eqn:psi-eta-t}
\psi_{\eta\left(t\right)}:\Omega\mapsto\Omega_{\eta\left(t\right)},\quad t\in I
\end{equation}
and defines a diffeomorphism (whose regularity depends on $\eta$). For details please see \cite{MS22,LR14}.
Note that, in particular, self-penetrations are excluded whenever $\norm{\eta}_{L^{\infty}_{t,x}}\leq L<L_0$.

\subsubsection*{The Koiter elastic energy}\label{subsection:Koiter}
Following \cite[Theorem 4.2.1]{Ci05}, we introduce the  
% \emph{linearized change of metric tensor} by 
%\begin{equation}\label{eqn:tensor-metric}
%\mathbb{G}_{ij}\left(\eta\right):=\left[\partial_{i}\phi_{\eta}\cdot\partial_{j}\phi_{\eta}-\partial_{i}\phi\cdot\partial_{j}\phi\right]^{lin},\quad i,j=1,2.
%\end{equation}
%
%We consider the \emph{normal direction to the deformed surface}  $\nu_{\eta}$ 
%\begin{equation}
%\nu_{\eta}:=\partial_{1}\phi_{\eta}\times\partial_{2}\phi_{\eta}.
%\end{equation}
%More explicitly, it takes the form 
%\[
%\begin{aligned}
%\nu_{\eta}= & \left(\partial_{1}\phi\times\partial_{2}\right)\phi+\left(\partial_{1}\phi\times\nu_{\phi}\right)\partial_{2}\eta+\left(\partial_{1}\phi\times\partial_{2}\right)\nu_{\phi}\eta+\\
% & +\left(\nu_{\phi}\times\partial_{2}\phi\partial_{1}\right)\eta+\left(\nu_{\phi}\times\partial_{2}\nu_{\phi}\right)\left(\partial_{1}\eta\right)\eta+\left(\partial_{1}\nu_{\phi}\times\partial_{2}\phi\right)\eta\\
% & +\left(\partial_{1}\nu_{\phi}\times\nu_{\phi}\right)\partial_{2}\eta+\left(\partial_{1}\nu_{\phi}\times\partial_{2}\nu_{\phi}\right)\eta^{2}
%\end{aligned}
%\]
%The \emph{linearized change of curvature tensor} is then defined as
%\begin{equation}\label{eqn:tensor-curvature}
%\mathbb{R}_{ij}^{\sharp}:=\left[\frac{\partial_{ij}\phi_{\eta}\cdot\nu_{\eta}}{\left|\partial_{1}\phi\times\partial_{2}\phi\right|}-\partial_{ij}\phi\cdot\nu\right]^{\text{lin}},\quad i,j=1,2.
%\end{equation}
%
%All-in-all,  if  $h$ denotes the thickness of the shell,  then the 
\emph{linearized Koiter energy} for $h$ the thickness of the plate $\eta$, with $\mathbb{G}(\eta)$ the linearised metric tensor and $\mathbb{R}_{ij}^{\sharp}$ as the linearised curvature tensor:
\begin{equation}\label{eqn:Koiter-energy}
K\left(\eta\right):=\sum_{i,j,m,l=1}^{2}\frac{h}{2}\int_{\omega}\mathbb{A}^{ij, ml}\mathbb{G}_{ml}\left(\eta\right)\mathbb{G}_{ij}\left(\eta\right)dA+\frac{h^{3}}{6}\int_{\omega}\mathbb{A}^{ij, ml}\mathbb{R}^{\sharp}_{ml}\left(\eta\right)\mathbb{R}^{\sharp}_{ij}\left(\eta\right)dA.
\end{equation}
Here $\mathbb{A}$ is a fourth-order tensor known as \emph{the elasticity (or stiffness) tensor} defined through
\begin{equation}\label{eqn:lame}
\mathbb{A}^{ij,ml}=\frac{4\lambda \mu}{\lambda+2\mu}a^{ij}a^{ml}+4\mu (a^{im}a^{jl}+a^{il}a^{jm})
\end{equation}
 and $\lambda,  \mu$ are the Lam\'{e} coefficients and $a^{ij}$ is the contravariant metric tensor associated to $M$.
In \cite[Theorem 4.4.2]{Ci05} it is shown that, in fact, that the $L^2$ gradient of $K$ takes the form 
\[
K^{\prime}\left(\eta\right)=m\Delta^{2}\eta+B\eta.
\]
where $m>0$ is a constant depending on the elastic material and $B$ a second order differential operator.
It is also proved (in the same reference \cite{Ci05}) that the Koiter energy is $H^2$-coercive, that is there exists a positive constant $c_0>0$ for which 
\begin{equation}\label{eqn:c0-coercivity}
K\left(\eta\right)\ge c_{0}\left\Vert \eta\right\Vert _{H^{2}\left(\omega\right)}^{2}\quad\text{ for all }\eta\in H_{0}^{2}\left(\omega\right)
.\end{equation}
\subsubsection*{The equations}
Suppose now that at each time $t$ the domain $\Omega_{\eta(t)}$ is filled with an \emph{incompressible fluid}, whose velocity is $\mathbf{u}$ and pressure $p$ and fulfills the Navier-Stokes system. We consider here the  so-called \emph{dynamic pressure condition} $\frac{1}{2}\left|\mathbf{u}\right|^{2}+p=P$ where $P:I\times \Gamma_p\mapsto\mathbb{R}$ is a prescribed, time-periodic function. Other types of inflow/outflow boundary conditions could be treated by the methodology developed here (see Remark~\ref{rem:boundary}). 

Therefore the fluid equations  take the form
\begin{equation}\label{eqn:fluid}
\begin{cases}
\rho_f(\partial_{t}\mathbf{u}+\left(\mathbf{u}\cdot\nabla\right)\mathbf{u})=\text{\text{div}\ensuremath{\mathbb{T}}+\ensuremath{\mathbf{f}} } & \text{in}\ I\times\Omega_{\eta}\\
\text{div}\mathbf{u}=0 & \text{in}\ I\times\Omega_{\eta}\\
\rho_f\frac{\left|\mathbf{u}\right|^{2}}{2}+p=P& \text{on}\ I\times\Gamma_{p}\\
u\cdot\tau_{i}=0 & \text{on}\ I\times\Gamma_{p}\text{ for }i\in\{1,2\}.\\
\mathbf{u}=\mathbf{0} & \text{on}\ I\times\Gamma_{D}\\
\mathbf{u}\circ\phi_{\eta}=\left(\partial_{t}\eta\right)\nu & \text{on}\ I\times\omega\\
\mathbf{u}\left(0,\cdot\right)=\mathbf{u}\left(T,\cdot\right) & \text{in}\ \left\{ 0\right\} \times\Omega_{\eta\left(0\right)}
\end{cases}
\end{equation}

Here $\rho_f$ represents the fluid (constant) density, 
$\mathbb{T}$ is the  \emph{Cauchy stress tensor}
 \begin{equation}\label{eqn:Cauchy-stress}
\mathbb{T}\left(\mathbf{u},p\right):=\sigma\cdot D\left(\mathbf{u}\right)-p\mathbb{I}_{3}
\end{equation}
and $D$ denotes the \emph{symmetric gradient} defined as
\begin{equation}
 D\left(\mathbf{u}\right):=\frac{1}{2}\left(\nabla\mathbf{u}+\left(\nabla\mathbf{u}\right)^{T}\right)
\end{equation}
while
$rho_f$ and $\sigma$ represent the fluid density and viscosity respectively.

Concerning now the shell displacement $\eta$, it solves   the wave-type equation
\begin{equation}\label{eqn:shell}
\begin{cases}
\rho_{S}h\partial_{tt}\eta+K^{\prime}\left(\eta\right)=g+\mathbf{F}\cdot\nu & \text{in}\ I\times\omega\\
\eta=\left|\nabla\eta\right|=0 & \text{on}\ I\times\partial\omega\\
\eta\left(0,\cdot\right)=\eta\left(T,\cdot\right) & \text{in}\ \omega
\end{cases}
\end{equation}
where $\rho_{S}>0$ denotes the density of the shell. 
Here
$K^{\prime}(\eta)$ denotes the $L^{2}$ gradient of the functional $K(\eta)$ from \eqref{eqn:Koiter-energy} and therefore
\[
K\left(\eta,\xi\right):=\langle K'(\eta),\xi\rangle = m\int_\omega\nabla^2\eta\cdot \nabla^2\xi\, dA+ \langle B\eta,\xi\rangle\text{ for all }\eta,\xi\in H_{0}^{2}\left(\omega\right).
\]
Further, $\mathbf{F}$ is the force exerted by the fluid on the boundary of the domain and is  defined as
\begin{equation}\label{eqn:force-fluid-shell-normal}
\mathbf{F}\left(t,\cdot\right):=-\mathbb{T}\left(\mathbf{u},p\right)\nu_{\eta\left(t\right)}\circ\phi_{\eta\left(t\right)}\left|\det \nabla\phi_{\eta\left(t\right)}\right|.
\end{equation}
Note that the definition $\mathbf{F}$ keeps track of possible dilations or contractions of the volume, since a change of variables would produce the term $\frac{\det\nabla\phi_{\eta\left(t\right)}}{\left|\det\nabla\phi_{\eta\left(t\right)}\right|}=\text{sign\ensuremath{\left(\det\nabla\phi_{\eta\left(t\right)}\right)}}$.
In the following we shall take for simplicity
\[
\sigma=2\text{ and }\rho_f=\rho_S h=1.
\]
Let us now write the coupled system formed by \eqref{eqn:fluid} and \eqref{eqn:shell} to obtain the \emph{fluid-structure interaction problem} which consists in the following system
\begin{equation}\label{eqn:system}
\begin{cases}
\partial_{t}\mathbf{u}+\left(\mathbf{u}\cdot\nabla\right)\mathbf{u}=\text{\text{div}\ensuremath{\mathbb{T}}+\ensuremath{\mathbf{f}} } & \text{in}\ I\times\Omega_{\eta}\\
\text{div}\mathbf{u}=0 & \text{in}\ I\times\Omega_{\eta}\\
\frac{\left|\mathbf{u}\right|^{2}}{2}+p=P\left(t\right) & \text{on}\ I\times\Gamma_{p}\\
\mathbf{u}\cdot{\bf \tau}_{i}=0 & \text{on}\ I\times\Gamma_{p}\text{ for }i\in\{1,2\}\\
\mathbf{u}=\mathbf{0} & \text{on}\ I\times\Gamma_{D}\\
\mathbf{u}\circ\phi_{\eta}=\left(\partial_{t}\eta\right)\nu & \text{on}\ I\times\omega\\
\mathbf{u}\left(0,\cdot\right)=\mathbf{u}\left(T,\cdot\right) & \text{in}\ \left\{ 0\right\} \times\Omega_{\eta\left(0\right)}\\
\partial_{tt}\eta+K^{\prime}\left(\eta\right)=g+{\bf F}\cdot\nu & \text{in}\ I\times\omega\\
\eta\left(0,\cdot\right)=\eta\left(T,\cdot\right),\partial_{t}\eta\left(0,\cdot\right)=\partial_{t}\eta\left(T,\cdot\right) & \text{in}\ \omega\\
\left|\eta\right|=\left|\nabla\eta\right|=0 & \text{on}\ I\times\partial\omega.
\end{cases}
\end{equation}
\subsection{Main results}
Our main result has a quantitative dependence on how curved the reference geometry is. For that we introduce $\kappa_1(\bfp)$ and $\kappa_2(\bfp)$ as the two principle curvatures of $M$ at the point $\bfp$. Further, to shorten notation we define $\kappa:=(\kappa_1,\kappa_2)$ containing the full curvature information of $M$.% Moreover we define
% \[
% \kappa:=\max_{\bfp\in M}\abs{\kappa_1(\bfp)}+\abs{\kappa_2(\bfp)}
% \]
 % as the maximal principal curvature of $M$. This is precisely the point where degeneration of the shell may happen most easily.

\begin{theorem}\label{thm:main}
Let $\Omega$ be a given reference domain with properties described as above. Then there exists a constant $\tilde{C}$ depending on $\Gamma_p,\Gamma_D,\kappa,L,\abs{M}$ and $c_0$ such that if
$\left(\mathbf{f},g,P\right)\in L_{\text{per}}^{2}\left(I;L^2(\mathbb{R}^{3})\right)\times L_{\text{per}}^{2}\left(I;L^2(\omega)\right)\times L_{\text{per}}^{2}\left(I;L^2(\Gamma_p)\right)$ satisfies
\begin{equation}
\left\Vert \mathbf{f}\right\Vert _{L_{t}^{2}L_{x}^{2}}+\left\Vert g\right\Vert _{L_{t}^{2}L_{x}^{2}}+\left\Vert P\right\Vert _{L_{t}^{2}L_{x}^{2}}\le\tilde{C}
\end{equation}
then there exists at least one weak time-periodic solution $\left(\mathbf{u},\eta\right)$ as it is defined in Definition~\ref{def:weak-soln}. 
Furthermore, it enjoys the \emph{diffusion estimate}
\begin{align*}\int_{0}^{T}\int_{\Omega_{\eta\left(t\right)}}\left|\nabla\mathbf{u}\right|^{2}\,dxdt\leq\int_{0}^{T}\int_{\Omega_{\eta\left(t\right)}}\mathbf{f}\cdot{\bf u}\,dxdt+\int_{0}^{T}\int_{\omega}g\partial_{t}\eta\,dAdt+\int_{0}^{T}\int_{\Gamma_{p}}P{\bf u}\cdot\nu\,dAdt
\end{align*}

and the {\em additional regularity estimate}
\begin{equation}\label{eqn:main-energy-estimate}
\sup_{t\in I}E\left(t\right)+\left\Vert \mathbf{u}\right\Vert _{L_{t}^{2}W_{x}^{1,2}}^{2}\leq C,
\end{equation}
with $C$ depending on $\tilde{C}$, $\Gamma_p,\Gamma_D,\kappa,L,\abs{M},c_0$ and $T$,
where $E(t)$ denotes here and in the rest of the paper
\begin{equation}
    E\left(t\right):=\frac{1}{2}\int_{\Omega_{\eta\left(t\right)}}\left|\mathbf{u}\right|^{2}dx+\frac{1}{2}\int_{\omega}\left|\partial_{t}\eta\right|^{2}dA+K\left(\eta\left(t,\cdot\right)\right),\quad t\in I.
\end{equation}
\end{theorem}
\begin{remark}[On the size of $\tilde{C}$ and C] The constant $C$ is made precise in Corollary~\ref{cor} and depends  quadratically on $\tilde{C}$.
We made quite an effort to make the dependence of $\tilde{C}$ on the reference geometry, the time interval and the Lam\'e constants precise. See again Corollary~\ref{cor} for the precise dependencies.

Important is that the constant $\tilde{C}$ that appears in Theorem~\ref{thm:main} has not to be small. The size is reduced first to ensure a-priori that a self-intersection is excluded and second to compensate for the curved reference geometry. Large forces are admissible if the Lam\'{e} coefficients of the shell are large in comparison to its thickness, which resembled the case of a stiff thick solid. Indeed, for $c_0\to \infty$, we find $\tilde{C}\to \infty$. Moreover, when $\kappa_1,\kappa_2\to 0$ in Corollary~\ref{cor}, the estimates relate to the estimates derived in \cite{our-paper} for the flat case. This means the dependencies only differ by the fact that here pressure boundary values are considered.
\end{remark}

\begin{remark}[On the size of $\tilde{C}$ and C] The constant $C$ is made precise in Corollary~\ref{cor} and depends  quadratically on $\tilde{C}$.
We made quite an effort to make the dependence of $\tilde{C}$ on the reference geometry, the time interval and the Lam\'e constants precise. See again Corollary~\ref{cor} for the precise dependencies.

Important is that the constant $\tilde{C}$ that appears in Theorem~\ref{thm:main} has not to be small. The size is reduced first to ensure a-priori that a self-intersection is excluded and second to compensate for the curved reference geometry. Large forces are admissible if the Lam\'{e} coefficients of the shell are large in comparison to its thickness, which resembled the case of a stiff thick solid. Indeed, for $c_0\to \infty$, we find $\tilde{C}\to \infty$. Moreover, when $\kappa_1,\kappa_2\to 0$ in Corollary~\ref{cor}, the estimates relate to the estimates derived in \cite{our-paper} for the flat case. This means the dependencies only differ by the fact that here pressure boundary values are considered.
\end{remark}
\begin{remark}[Case $\Gamma_p=0$]
\label{rem:noslip}
In case no inflow or outflow is present, the mean-value of the pressure is determined by the shell. Hence the mean-value of the pressure is bound to depend on the Koiter energy. It is yet an open problem, how to derive an a-priori estimate in that case. An exception is the case, when $M$ is flat. This case was treated in \cite{our-paper}.
\end{remark}
\begin{remark}
\label{rem:boundary}
The boundary  condition $\frac{1}{2}\left|{\bf u}\right|^{2}+p=P_{\text{in/out }}$ representing the prescribed inflow/outflow is just one possible choice to model inflow/outflow.  Another alternative would have been a so called "do-nothing" boundary condition which would consist in imposing $\mathbb{T}\nu=\frac{1}{2}\left|{\bf u}\right|^{2}\nu\ \text{on}\ \Gamma_{\text{out}}$. This boundary condition was rigorously justified in \cite{BaSt21}.
\end{remark}
%\begin{remark}
%\label{rem:nonnew}
% Due to the construction of the Galerkin approximation, it seems that the case of non-Newtonian fluids of $p$-type with $p \ge 2$ can be covered. This means that instead of the term $\Delta \mathbf{u}$, we consider the term $\Delta_{p}=\text{div}\left(\left|\nabla\mathbf{u}\right|^{p-2}\left|\nabla\mathbf{u}\right|\right)$. See \cite{Lengeler-p-fluid} for further details.
% \end{remark}
    Let us observe that our methodology also provides a proof of the following result on {\em initial value problems} including inflow/outflow boundary values. To our best knowledge, it was not considered in the literature; compare with \cite[Theorem 3.5]{LR14}, where the respective initial value problem with everywhere no-slip boundary values was considered. %The proof is actually slightly simplified and  follows the same lines as of Theorem~\ref{thm:main}.
    \begin{theorem}
    \label{thm:cauchy}
        Let
        \[
\left(\mathbf{f},g,P\right)\in L_{\text{loc}}^{1}\left(\left[0,\infty\right);L_{\text{loc}}^{2}\left(\mathbb{R}^{3}\right)\right)\times L_{\text{loc}}^{1}\left(\left[0,\infty\right);L_{\text{loc}}^{2}\left(\omega\right)\right)\times L_{\text{loc}}^{2}\left(\left[0,\infty\right)\right)
        \]
 and let the initial-data 
 \[
\left(\mathbf{u}_{0},\eta_{0},\eta_{1}\right)\in L^{2}\left(\Omega_{\eta\left(0\right)}\right)\times H_{0}^{2}\left(\omega\right)\times L^{2}\left(\omega\right)
 \]
 be such that $\text{div}\mathbf{u}_{0}=0$ and $\text{tr}_{\eta_{0}}\mathbf{u}_{0}=\eta_{1}\nu$. Then, there exists a time $0<T^{*}\le \infty$ such that the corresponding initial-value problem admits a weak solution satisfying the usual energy inequality in $[0,T^{*}]$. The time $T^*<\infty$ only in case of topological changes of the boundary.
    \end{theorem}

\subsection{Mathematical strategy and technical novelties}
\label{sec:math}
When shells are considered the control to exclude geometric/topological degeneracy is much more subtle then in the previously considered case of  plates~\cite{our-paper}.
In particular it is necessary to evaluate carefully the connection between the size of the forcing and its potential impact on the geometry. This reveals on a technical level in form of non-quadratic terms that can only be estimated via a careful understanding of their meaning. The control of their appearance is the main achievement of the present paper:
See Proposition~\ref{prop:formal-decoupled-estimates} and Corollary~\ref{cor}, where precise dependencies are introduced that do allow to close the circle of estimates. The necessary technical tools are a here introduced direct trace estimate, Lemma~\ref{lem:trace} and a refined divergence-free extension operator, introduced in \cite{MS22}. The operator allows to extend {\em any} functions defined on the moving boundary to the whole fluid domain by a divergence free function. This is possible as a (smooth) inflow/outflow is created to make the normal trace zero. Hence, the operator creates coupled test functions, suitable for the weak formulation. This operator became a central object in the theory of fluid-structure interactions~\cite{breit-schwarz-fourier,kreml2023time}. Hence the innovations here might be of independent interest,  see Proposition~\ref{prop:Solenoidal-extension-o}.

The main technical achievement of the present paper is to develop a methodology that allows to perform the essential parts of the construction, including all necessary a-priori estimates, on the Galerkin level. We perform the following steps on the discrete level.
\begin{enumerate}
\item {\bf We decouple the problem.} This means for a {\em given} function $\delta:[0,T]\times \omega\to [-L,L]$ that is defining the time-changing fluid-domain, we seek coupled solutions. We linearize the convective term in a natural why in dependence of $\delta$.
\item {\bf We solve the Cauchy-problem} for a given initial data for the linear decoupled problem. This follows by standard ODE theory
\item {\bf Fixed point one:} For each given geometry $\delta$ we find an initial data that provides a time-periodic solution. Here we use Sch\"affer's fixed point theorem. It relays on a-priori estimates that are only possible due to the special construction of the Galerkin bases. 
\item {\bf Fixed point two:} We map a time-periodic geometry to its set of time-periodic solutions and find a fixed point there. Here we use the set-valued fixed point of Kakutani-Glicksberg-Fan.
\end{enumerate}
We wish to emphasize that the a-priori estimates necessary for each step are obtained with great care, such that the argument closes. The advantage of the here introduced methodology is that one can perform the fixed points on the Galerkin level {\em using the special structure of the basis}. This seems to us a natural approach to the analysis of weak solutions to fluid-structure interactions, {\em as the space of test-function is an essential part of the solution}. 

Hence the construction of a suitable Galerkin basis is one of the technical innovations of the paper. To some extend it follows the ideas of~\cite{LR14,breit-schwarz-fourier}, as we are deviating between the part of the fluid-velocity that relates to the boundary deformation and its bulk part. Both parts of the basis will eventually depend on the solution, as is typical for fluid-structure interactions. A suitable bases for the here considered purposes is chosen as follows.
 \begin{itemize}
 \item The part of the Galerkin bases that is related to the deforming domain is constructed explicitly via the solenoidal extension operator Proposition~\ref{prop:Solenoidal-extension-o}: A suitable basis for the solid equation is extended into the fluid domain.
 \item The second part of the Galerkin basis has a zero trace at the moving boundary part. Here we use a solenoidal push-forward of the basis of eigenvalues of the Stokes operator on the reference domain (following the ideas of~\cite{LR14}). Hence again the basis is critically depending on the geometry and hence the solution. Moreover, the setting of mixed boundary conditions excludes a strong a-priori regularity for the eigenvalues of the respective Stokes operator, even in smooth domains. Accordingly we rely on abstract results for eigenvalues of self-adjoined operators in Hilbert spaces, which means that the control is reduced to $H^1$ and $L^2$.
 \end{itemize}
What might be considered as a disadvantage is that the limit passage of a non-linear sequence of Galerkin solution has to be undertaken. This seems to us to stay a debatable point. In particular as the non-linear limit passage involves the $L^2$ compactness of a sub-sequence of solution, which is a quite involving argument in the setting of fluid-structure interactions. Nevertheless we undertook the effort here to show that the compactness is true, which turned out to be substantially more technical then for sequences of continuous solutions, see Subsection~\ref{sec:L2comp}. As it works for a basis with rather weak regularity assumption, we believe it a valid contribution for the future, as it allows to use the advantages to perform a-priori analysis on an explicit and well chosen Galerkin basis. This advantage seemed critical to us in the present setting and we believe it to be very helpful for more future applications whenever time-changing domains are considered.

\subsection{Overview of the paper}
In Section~\ref{sec:prelim} we introduce the moving domains, the divergence-free extension operator $\mathcal{F}_{\delta}$, an explicit trace estimate and introduce the definition of the time-periodic weak solution.
 
Section~\ref{sec:decoupled-problem} is dedicated to establish the key a-priori estimates. They are contained in Proposition~\ref{prop:formal-decoupled-estimates}, Corollary~\ref{cor} and Corollary~\ref{cor:gal}. This means introducing according suitable decoupled systems and to show how the geometric characteristics influencing the estimates can be controlled.
The paper ends with Section~\ref{sec:Proof-main-thm} where we perform first the existence of a stabilized coupled solution, by the strategy indicated above. In the final Subsection~\ref{sec:L2comp} it is shown how this result implies the proof of Theorem~\ref{thm:main} and Theorem~\ref{thm:cauchy}.

\section{Preliminaries}\label{sec:prelim}

\subsection{Function spaces on variable domains}\label{subsection:function-spaces}
For admisible domains $\Omega_{\eta}$, that is for which $\left\Vert \eta\right\Vert _{L_{t,x}^{\infty}}<L$,
we introduce the Lebesgue and Sobolev spaces adapted to the moving domains $\Omega_{\eta}$ are defined as follows: for $1 \le p, \ q \le \infty$ we have
\begin{equation}
\begin{aligned}L^{p}\left(I;L^{q}\left(\Omega_{\eta}\right)\right):= & \left\{ \mathbf{v}\in L^{1}\left(I\times\Omega_{\eta}\right):\mathbf{v}\left(t,\cdot\right)\in L^{q}\left(\Omega_{\eta\left(t\right)}\right)\ \text{for a .e.}\ t\in I,\right.\\
 & \left\Vert \mathbf{v}\left(t,\cdot\right)\right\Vert _{L^{q}\left(\Omega_{\eta\left(t\right)}\right)}\in L^{p}\left(I\right)\left.\right\} \\
L^{p}\left(I;W^{1,q}\left(\Omega_{\eta}\right)\right):= & \left\{ \mathbf{v}\in L^{1}\left(I\times\Omega_{\eta}\right):\mathbf{v}\left(t,\cdot\right)\in W^{1,q}\left(\Omega_{\eta\left(t\right)}\right)\ \text{for a .e.}\ t\in I,\right.\\
 & \left\Vert \nabla \mathbf{v}\left(t,\cdot\right)\right\Vert _{L^{q}\left(\Omega_{\eta\left(t\right)}\right)}\in L^{p}\left(I\right)\left.\right\}. 
\end{aligned}
\end{equation}
One can now introduce the corresponding \emph{time-periodic} spaces of functions $L_{\text{per}}^{p}\left(I;L^{q}\left(\Omega_{\eta}\right)\right)$ and resp. $L_{\text{per}}^{p}\left(I;W^{1,q}\left(\Omega_{\eta}\right)\right)$ obtained by taking the closure of the set $\left\{ \varphi\in C^{\infty}\left(I\times\mathbb{R}^{3}\right):\varphi\left(0,\cdot\right)=\varphi\left(T,\cdot\right)\right\} $ in the respective norms.

Then, given a time-periodic moving domain $\Omega_{\delta}$ determined by a displacement   $\delta\in C^{3}\left(\omega\right)$ with $\left\Vert \delta\right\Vert _{L_{t,x}^{\infty}}<L$ 
we can introduce the \emph{space of solutions} $V_{S}^{\delta}$  by
\begin{align*}V_{S}^{\delta}:= & \Big\{ \left(\mathbf{u},\eta\right)\in L_{\text{per}}^{2}\left(I;H_{\text{div}}^{1}\left(\Omega_{\delta}\right)\right)\cap L_{\text{per}}^{\infty}\left(I;L^{2}\left(\Omega_{\delta}\right)\right)\times L_{\text{per}}^{\infty}\left(I;H_{0}^{2}\left(\omega\right)\right)\cap W_{\text{per}}^{1,\infty}\left(I;L^{2}\left(\omega\right)\right):
\\
& \quad \mathbf{u}=0\text{ on }\Gamma_{D},\,\mathbf{u}\cdot\tau_{1}=\mathbf{u}\cdot\tau_{2}=0\text{ on }\Gamma_{p},\mathbf{u}\circ\phi_{\delta}=\partial_{t}\eta\nu\text{ on }\omega\Big\} 
\end{align*}
and the corresponding \emph{space  of test-functions} $V_{T}^{\delta}$ by
\begin{align*}
V_{T}^{\delta}:= & \Big\{ \left(\mathbf{q},\xi\right):\mathbf{q}\in W_{\text{per}}^{1,2}\left(I;H_{\text{div}}^{1}\left(\Omega_{\delta}\right)\right)\, :\, \mathbf{q}=0\text{ on }\Gamma_{D},\mathbf{q}\cdot\tau_{1}=\mathbf{q}\cdot\tau_{2}=0\text{ on }\Gamma_{p},
 \\
 & \xi\in W_{\text{per}}^{1,2}(I,L^2(\omega))\times L^2_{\text{per}}\left(I;H_{0}^{2}\left(\omega\right)\right),\mathbf{q}\circ\phi_{\delta}=\xi\nu\text{ on }\omega\Big\} 
\end{align*}
Their definition will become more meaningul after introducing our notion of solution in Subsection~\ref{subsection:weak-soln}.

\subsection{Some analytical tools}
First, we  point out the existence of the \emph{Piola mapping}: for any  $\varphi: \Omega \mapsto\mathbb{R}^{3}$ we can define its extension to the whole fluid domain $\Omega_{\delta(t)}$ by the formula
\[
\mathcal{J}_{\delta}\varphi:=\left(\nabla\psi_{\delta}\left(\det \nabla\psi_{\delta}\right)^{-1}\varphi\right)\circ\psi_{\delta}^{-1}\]
If $\delta$ is smooth, the mapping
 $\mathcal{J}_{\delta}$ defines an isomorphism between the Lebesgue and Sobolev spaces on $\Omega$ and the corresponding ones on $\Omega_\delta$  and moreover preserves the zero boundary values and the divergence-free constraint.

In order to  rigorously justify the traces  of  functions  defined on $\Omega_{\delta(t)}$ one may use the following \emph{trace operator}:
If $1 < p \le \infty$, then for any $r \in (1,p)$ the mapping 
\[
\text{tr}_{\delta}:W^{1,p}\left(\Omega_{\delta}\right)\mapsto W^{1-\frac{1}{r},r}\left(\omega\right),\quad\text{tr}_{\delta}\left(v\right):=\left(v\circ\psi_{\delta}\right)_{|\omega}
\] is well defined and continuous, with continuity constant depending on $\Omega, r, p,\delta$; for a proof see~\cite{LR14}. We provide here a trace theorem, that depends on the reference domain $\Omega$.
\begin{lemma}[Trace operator]
\label{lem:trace}
Assume that $\delta\in C^0_0(\omega)$ with $\norm{\delta}_\infty\leq L<L_0$ . Then $v\in C^0(\Omega_\delta)\cap W^{1,p}(\Omega)$ given by $v(\bfp+\delta\nu(\bfp))=\beta(\bfp)\nu(\bfp)$ on $M$ satisfies the estimate
\[
\norm{\beta}_{L^p(M)}\leq c(\kappa,L,L_0-L)\norm{ v}_{W^{1,p}(\Omega_\delta)},
\]
where $\kappa$ and $L$ are geometric properties of $M$ defined in the introduction.
\end{lemma}
\begin{proof}
We use the coordinates $x=\bfp(x)+s(x)\nu(\bfp(x))$
%, where $\hat{v}(s,\bfp):=v(\bfp+s\nu(\bfp))$. Then
We take $\ell<L_0-L$, then we find 
\begin{align*}
\int_M \abs{\beta(\bfp)}d\bfp 
&\leq 
\fint_{-\ell}^0\int_M \abs{\beta(\bfp)-v(\bfp+(s+\delta(\bfp))\nu(\bfp))}d\bfp\, ds
+\fint_{-\ell}^0\int_M \abs{v(\bfp+(s+\delta(\bfp))\nu(\bfp))}d\bfp\, ds
\\
&\leq \fint_{-\ell}^0\int_M \absB{\int_{-s}^0\partial_\nu v(\bfp+(\sigma+\delta(\bfp))\nu)\cdot\nu(\bfp)d\sigma}d\bfp\, ds +\frac{L+\ell}{\ell}\fint_{-L-\ell}^0\int_M \chi_{\Omega_\delta}\abs{v(\bfp+s\nu(\bfp))}d\bfp\, ds
\\
&\leq \int_M \int_{-L-\ell}^0\chi_{\Omega_\delta}\abs{\nabla v(\bfp+s\nu(\bfp))}d\bfp\, ds
+c(L,\kappa)\int_{\Omega_\delta} \abs{v} d\bfp\, ds
\\
&\leq c(L,\kappa,L_0-L)\norm{v}_{W^{1,1}(\Omega_\delta)}.
\end{align*}
The $p$-case follows by taking $v=\abs{v}^p$. Indeed, then by Young's inequality
\[
\int_M \abs{\beta(\bfp)}^p d\bfp \leq c(\kappa,L)\norm{\abs{v}^p}_{W^{1,1}(\Omega_\delta)}\leq c(\kappa,L)\norm{v}_{W^{1,p}(\Omega_\delta)}^p. 
\]
\end{proof}

Next, we  recall the classical  Reynolds' Transport Theorem: 
\begin{theorem}\label{thm:Reynolds}
For all $g=g(t,x)$ such that the following quantities are smooth, it holds that 
\begin{equation}
    \frac{d}{dt}\int_{\Omega_{\eta\left(t\right)}}gdx=\int_{\Omega_{\eta\left(t\right)}}\partial_{t}gdx+\int_{\partial\Omega_{\eta\left(t\right)}}g\mathbf{v}\cdot\nu_{\eta}dA.
\end{equation}
\end{theorem}
Here $\mathbf{v}$ is the speed of the boundary $\partial\Omega_{\eta(t)}$, so $\mathbf{v}\left(t,\cdot\right)=\left(\partial_{t}\eta\nu\right)\circ\phi_{\eta\left(t\right)}^{-1}$.
 
Then, by \cite[Lemma A.5]{LR14}
we  have  Korn's identity
\begin{equation}
\int_{\Omega_{\eta\left(t\right)}}D\left(\mathbf{u}\right):D\left(\mathbf{q}\right)dx=\frac{1}{2}\int_{\Omega_{\eta\left(t\right)}}\nabla\mathbf{u}:\nabla\mathbf{q}dx.
\end{equation}

Finally, it will be useful to consider the identity 
\begin{equation}
\label{eq:cov}
    \int_{\partial\Omega_{\eta}\cap \psi_\eta(M)}\left(f\nu\circ\phi_{\eta}^{-1}\right)\cdot\nu_{\eta}dA=\int_{\omega}fJ_{\eta}dA\quad\text{ for all } f\in L^{1}\left(\omega\right)
\end{equation}
where $J_{\eta}(t,x):=G\eta^{2}-2H\eta+1
$ 
with $G=\kappa_1 \cdot \kappa_2$ and $H=\kappa_1+\kappa_2$ denoting the principal curvatures of $\partial\Omega$. See \cite[Proposition 2.11]{LR14} for detailed computations.

Next, since we will perform several fixed-points, we shall need beyond the classical Leray-Schauder/Sch\"{a}ffer Theorem (see for example \cite[Section 9.2.2., Theorem 4]{evans}, also the following set-valued fixed-point result which can be found in \cite[Chapter 2, Section 5.8]{GD03}:
\begin{theorem}[Kakutani-Glicksberg-Fan]\label{thm: Kakutain} 
Let $C$ be a convex subset of a normed vector space $Z$ and let $F:C \to \mathcal{P}(C)$ be a  set-valued mapping which has closed graph. Moreover, let $F(C)$ be contained in a compact subset of $C$, and let $F(z)$ be non-empty, convex, and compact for all $z \in C$. Then $F$ possesses a fixed point, that is there is $c_0 \in C$ with $c_0 \in F(c_0)$. 
\end{theorem}
We say a set-valued mapping $F : C \mapsto \mathcal{P}(C)$ has \emph{closed graph} provided that the set $\left\{ \left(x,y\right):y\in F\left(x\right)\right\} $ is closed in $X \times Y$ with the product topology or that, equivalently, for any sequences $x_n \to x$ and $y_n \to y$ with $y_n \in F(x_n)$ for any $ \ge 1$ it follows that $y \in F(x)$.

The following theorem of Aubin-Lions type is taken from~\cite[Theorem 5.1]{MS22}
\begin{theorem}
\label{thm:auba}
{Let $X,Z$ be two Banach spaces, such that $X'\subset Z'$.
Assume that $f_n:(0,T)\to X$ and $g_n: (0,T)\to X'$}. Moreover assume the following: 
\begin{enumerate}
\item The {\em weak convergence}: for some $s\in [1,\infty]$ we have that $f_n\toweakstar f$ in $L^s(X)$ and $g_n\toweakstar g$ in $L^{s'}(X')$.
%\item The {\em uniform bound} on one sequence
%\[
%\sup_n\norm{f_n}_{L^p(0,T;Y)}\leq C.
%\]
%moreover we suppose that $f_n\toweakstar f$ in $L^p(0,T;X)$.
\item The {\em approximability-condition} is satisfied: For every $m\in \N$ there exists a $f_{n,m}\in L^s(0,T;X)\cap L^1(0,T;Z)$,
%
%and every $t\in [0,T]$ there exists a mollifying operator $\bor{X\ni}\phi\mapsto (\phi)_{\delta,t}\ni Z$.
such that for every $\varepsilon\in (0,1)$ there exists a $m_\epsilon\in \N$ (depending only on $\varepsilon$) such that %for a.e. $t\in [0,T]$
\[
%\sup_{k,n}
\norm{f_n-f_{n,m}}_{L^s(0,T;X)}\leq \epsilon\text{ for all } m>m_0%\text{ and }\norm{f_{n,\delta}}_{Z}\leq  C(\delta)(1+\norm{f_n}_Y^p).
\]
and for every $m\in \N$ there is a $C(m)$ such that
\[
\norm{f_{n,m}}_{L^1(0,T;Z)}\, dt\leq C(m).
\]
Moreover, we assume that for every $m$ there is a function $f_m$, and a subsequence such that $f_{n,m}\toweakstar f_m$ in $L^s(X)$.
\item
The {\em equi-continuity} of $g_n$. We require that there exists an $\alpha\in (0,1]$, {a sequence $A_n$ that is uniformly bounded in $L^1([0,T])$}, such that for every $m\in \N$ that there exist a $C(m)>0$ and an $n_m\in \N$ such that for $\tau>0$ and a.e.\ $t\in [0,T-\tau]$
\[
\sup_{n\geq n_{m}} \absB{\fint_{0}^\tau\skp{g_n(t)-g_n(t+s),f_{n,m}(t)}_{X',X}\, ds} \leq C(\delta)\tau^\alpha(A_n(t)+1).
\]
\item The {\em compactness assumption} is satisfied: $X'\hookrightarrow \hookrightarrow Z'$. More precisely, every uniformly bounded sequence in $X'$ has a strongly converging sub-sequence in $Z'$. 
\end{enumerate}
Then there is a subsequence, such that
\[
\int_0^T\skp{f_n,g_n}_{X,X'}\, dt\to \int_0^T\skp{f,g}_{X,X'}\, dt.
\]
\end{theorem}

\subsection{Weak solutions}
\label{subsection:weak-soln}
We introduce here a notion of weak solution for the coupled system~\eqref{eqn:system}. Assume all the appearing quantities are smooth and let $(\mathbf{q}, \xi) \in V_{T} ^{\eta}$ be a test-function.
By multiplying the fluid equation \eqref{eqn:fluid} by $\bf{q}$ and the shell equation  \eqref{eqn:shell} by $\xi$, after performing the integration by parts, taking into consideration the coupling between $\mathbf{u}$ and $\partial_{t}\eta$ and using Reynolds' transport theorem (see Theorem~\ref{thm:Reynolds}), after
 summing up the results  we arrive at the following:
\begin{definition}\label{def:weak-soln}
We call a pair $(\mathbf{u}, \ \eta)\in V^{\eta}_{S}$  with $\left\Vert \eta\right\Vert _{L_{t,x}^{\infty}}<L$ a \emph{time-periodic weak solution} of the system \eqref{eqn:system} provided that  the following relation holds:

\begin{equation}\label{eqn:weak-form}
\begin{aligned}\int_{0}^{T}\int_{\Omega_{\eta\left(t\right)}}-{\bf u}\cdot\partial_{t}{\bf q}+\frac{\left({\bf u}\cdot\nabla\right){\bf u}\cdot{\bf q}-\left({\bf u}\cdot\nabla\right){\bf q}\cdot{\bf u}}{2}+\nabla{\bf u}:\nabla{\bf q}dxdt & +\\
\int_{0}^{T}\int_{\Gamma_{p}}P{\bf q}\cdot\nu\ dAdt+\int_{0}^{T}\int_{\omega}-\frac{1}{2}\left(\partial_{t}\eta\right)^{2}\xi J_{\eta}dAdt+\int_{0}^{T}\int_{\omega}\partial_{t}\eta\partial_{t}\xi dAdt+\int_{0}^{T}K\left(\eta,\xi\right)dt & =\\
\int_{0}^{T}\int_{\Omega_{\eta\left(t\right)}}{\bf f}\cdot{\bf q}dxdt+\int_{0}^{T}\int_{\omega}g\xi dAdt\quad\text{ for all }\left({\bf q},\xi\right)\in V_{T}^{\eta}
\end{aligned}
\end{equation}
and that the energy inequality 
\begin{equation*}%\label{eqn:energy}
\frac{d}{dt}E\left(t\right)+\int_{\Omega_{\eta\left(t\right)}}\left|\nabla\mathbf{u}\right|^{2}dx+\int_{\Gamma_{p}}\mathbf{u}\cdot\nu P dA\leq \int_{\Omega_{\eta\left(t\right)}}\mathbf{f}\cdot\mathbf{u}dx+\int_{\omega}g\partial_{t}\eta dA
\end{equation*}
 is satisfied distributionally in time.
\end{definition}
Please notice that in case a smooth weak solution of Definition~\ref{def:weak-soln} possesses sufficient regularity it is a strong solution to \eqref{eqn:system}.

\subsection{A divergence-free extension operator}

In the present paper we need the following modification of the solenoidal extension operator introduced in \cite{MS22} (compare also with \cite{LR14}) adapted to the inflow/outflow setting. 
Its important properties are summarized in the following Proposition~\ref{prop:Solenoidal-extension-o}.
\subsection{Geometric setting}

Recall from Section~\ref{sec:math}
that there exists a maximal $L=L(\Omega)>0$ such that 
for $x\in \setR^n$ with $\text{dist}\left(x,\partial\Omega\right)<L$,   the  following quantities are well defined:
\[
\mathbf{p}\left(x\right):=\mathrm{argmin}_{y\in\partial\Omega}\left|y-x\right|\quad,s\left(x\right):=\begin{cases}
-\left|{\bf p}\left(x\right)-x\right| & x\in\Omega\\
\left|{\bf p}\left(x\right)-x\right| & x\in\mathbb{R}^{3}\setminus\Omega.
\end{cases}
\]
Let us set 
\[S_{L}:=\left\{ x\in\mathbb{R}^{3}:d\left(x,\partial\Omega\right)<L\right\} ,\quad\Omega_L:=\Omega\cup S_{L}\]
and recall that $\kappa=(\kappa_1, \kappa_2)$ represent the principal curvatures of $M$. Then, we have:
\begin{proposition}[Solenoidal extension]
\label{prop:Solenoidal-extension-o}
Let $\delta\in L^\infty(0,T;W^{1,2}(\omega))$ be  such that $\left\Vert \delta\right\Vert _{L_{t,x}^{\infty}}\le L<L_0$.
Then there exists a linear solenoidal extension operator 
\[
\testd: L^1(0,T;W^{1,1}_0(\omega))\to L^{1}(0,T;W^{1,1}(\Omega_L))
\]
 such that  $\divergence{\testd}(\xi)=0$ on $\Omega_\delta$ and 
$(\testd(\xi),\xi)\in V_{T}^{\delta}$ for $\xi \in W^{1,\infty}(0,T;H^2_0(\omega))$.

Moreover, for $q\in [1,\infty]$, $p\in (1,\infty)$ it satisfies the following estimates
\begin{equation}
\begin{aligned}\|\mathcal{F}_{\delta}\left(\xi\right)\|_{L^{q}(0,T;L^{p}(\Omega_L)} & \leq C(L,\kappa)\|\xi\|_{L^{q}(0,T;L^{p}(\omega))}\\
\|\mathcal{F}_{\delta}\left(\xi\right)\|_{L^{q}(0,T;W^{1,p}(\Omega_L))} & \leq C(L,\kappa)\Big(\|\xi\|_{L^{q}(0,T;W^{1,p}(\omega)}+\|\left|\xi\right|\left|\nabla\delta\right|\|_{L^{q}(0,T;L^{p}(\omega)}
\\
&\quad +C(L,\kappa,\norm{\nabla \kappa_1}_\infty,\norm{\nabla \kappa_2}_\infty)\|\xi\|_{L^{q}(0,T;L^{p}(\omega)}\Big)\\
\|\partial_{t}\mathcal{F}_{\delta}\left(\xi\right)\|_{L^{q}(0,T;L^{p}(\Omega_L))} & \leq C(L,\kappa)\Big(\|\partial_t\xi\|_{L^{q}(0,T;L^{p}(\Omega_L))} + \|\left|\xi\right|\left|\partial_t\delta\right|\|_{L^{q}(0,T;L^{p}(\omega)}\Big)
\end{aligned}
\end{equation}
Moreover, the extension has non-trivial boundary values on $\Gamma_p$, which are estimated by the above estimates and the trace theorem.

Further, the dependence on $L$ of the constant $c(\kappa,L)$ is decreasing with $\kappa\to0$. In particular, in the flat case $\kappa=0$ there is no dependence on $L$ in the above estimates.\footnote{The quantity $\kappa$ is defined above Theorem~\ref{thm:main} and is related to the principal curvatures of $M$.} Indeed the restriction of $\norm{\delta}_{L^\infty}$ is reduced to a topological restriction without quantitative impact.

%\Big(\|\partial_{t}\xi\|_{L^{q}(0,T;L^{p}(\omega)}+\norm{\kappa}_\infty\|\left|\xi\right|\left|\partial_{t}\delta\right|\|_{L^{q}(0,T;L^{p}(\omega))}\Big),\\
%\|\nabla^{2}\mathcal{F}_{\delta}\left(\xi\right)\|_{L^{q}(0,T;L^{p}(\Omega_L)} & \leq C\Big(\|\nabla^{2}\xi\|_{L^{q}(0,T;L^{p}(\omega)}+\|\left|\xi\right|\left|\nabla^{2}\delta\right|\|_{L^{q}(0,T;L^{p}(\omega))}\Big)+\\
% & \quad C\Big(\|\left|\nabla\xi\right|\left|\nabla\delta\right|\|_{L^{q}(0,T;L^{p}(\omega))}+\|\left|\xi\right|\left|\nabla\delta\right|^{2}\|_{L^{q}(0,T;L^{p}(\omega))}\Big)\\
%\|\partial_{t}\nabla\mathcal{F}_{\delta}\left(\xi\right)\|_{L^{q}(0,T;L^{p}(\Omega_L)} & \leq C\Big(\|\partial_{t}\nabla\xi\|_{L^{q}(0,T;L^{p}(\omega)}+\|\left|\xi\right|\left|\nabla\partial_{t}\delta\right|\|_{L^{q}(0,T;L^{p}(\omega))}%+\|\left|\nabla\xi\right|\left|\partial_{t}\delta\right|\|_{L^{q}(0,T;L^{p}(\omega))}
%\Big)+\\
% & \quad+C\|\left|\partial_{t}\xi\right|\left|\nabla\delta\right|+\left|\nabla\xi\right|\left|\partial_{t}\delta\right|+\left|\xi\partial_{t}\delta\right|\left|\nabla\delta\right|\|_{L^{q}(0,T;L^{p}(\omega))}
%\end{aligned}
%\end{equation}
%whenever the right hand side is finite. Here $C$ depends only on $L$ and  the domain.
%Moreover, for $q^\sharp=\frac{2q}{3-q}$, if $q<3$; $q^\sharp<\infty$, if $q=3$ and $q^\sharp=\infty$, if $q>3$
%\begin{align*}
%\norm{\testd{\xi}}_{L^q(0,T;L^{p^\sharp}(\Gamma_p)}\leq C\Big( \|\xi\|_{L^{q}(0,T;W^{1,p}(\omega)}+\|\abs{\xi} \abs{\nabla \delta} \|_{L^{q}(0,T;L^{p}(\omega)}\Big).
%\end{align*}
\end{proposition}

\begin{proof}
The construction follows along the lines of \cite[Proposition~3.3]{MS22}.
With the coordinates $s(x),\bfp(x)$ introduced above we have that
\[
\nabla s(x)=\partial_{\bn} s(x)\bn=\bn,\text{ and }\nabla \bfp(x)=(\partial_{\tau_i(\bfp(x))}\bfp(x)%\cdot\tau_i
)_{i=1,2}.
\]
Further, we find
\[
\abs{\text{div}(\nu(\bfp(x))}\leq c(L,\kappa)
\text{ and }\abs{\nabla\text{div}(\nu(\bfp(x))}\leq c(L,\kappa,\norm{\nabla \kappa_1}_\infty,\norm{\nabla \kappa_2}_\infty)\text{ for }x\in S_L.
\]
Let us fix $L<L_1<L_0$ and a cut-off function 
\begin{equation}
\sigma_{L}\in C_{0}^{\infty}\left(\mathbb{R};\left[0,1\right]\right),\ \sigma_{L}\left(s\right)=\begin{cases}
1, & \left|s\right|\le L\\
0, & \left|s\right|\geq L_1
\end{cases}
\end{equation}
For $\xi \in H_0^{2}(\omega)$ we define its extension $\tilde{\xi}$ to $\partial \Omega$ through
\[\tilde{\xi}:=\begin{cases}
\left(\xi\nu\right)\circ\phi_{\delta}^{-1} & \text{on}\ M\\
-\psi_{p}\int_{\omega}\xi J_{\delta}dA\cdot\nu & \text{on}\ \Gamma_{p}\\
0 & \text{on}\ \Gamma_{D}
\end{cases}
\]
where $\psi_{p}\in C_{0}^{\infty}\left(\Gamma_{p}\right),\ \int_{\Gamma_{p}}\psi_{p}dA=1$.
Now we can set
\begin{equation}\label{eqn:def-F-delta}
\overline{\test}(\xi)(x):=e^{(\eta(\phi^{-1}(\bfp(x))-s(x)))\divg(\bn(\bfp(x))}\tilde{\xi}(\bfp(x))\sigma_L(s(x))\bn(\bfp(x)).
\end{equation}
Please note that for $x\in \partial\Omega_\eta$ we have that
\[
\overline{\test}(\xi)(x)=\bn(\bfp(x))\xi( y(x)),\;{x\in\partial\Omega_{\eta}}.
\] 
Then, for $x \in S_{L}$
 we have, following the same computations as in \cite[Proposition 3.3]{MS22} that
  \[
\text{div}_{x}\overline{\mathcal{F}}_{\delta}\left(\xi\right)=0\]
On $\Omega\setminus S_{L}$ we find that
\[
\divg(\overline{\test}(\xi)(x)) = e^{(\eta(t, (\phi^{-1}(\bfp(x)))-s(x))\divg(\bn(\bfp(x)))}\tilde{\xi}(\bfp(x))\sigma_L^{\prime}(s(x)),
\]
which has support inside $S_{L_1}$. 
We need to correct $\overline{\mathcal{F}}_{\delta}\left(\xi\right)$ within the set $\Omega\setminus S_{L}$ and by employing the Bogovski operator -see \cite{Bog}- we can solve the problem \[\begin{cases}
\text{div} \ \mathbf{v}=\text{div \ensuremath{\overline{\mathcal{F}}_{\delta}}} & \text{in}\ \Omega\setminus S_{L_1}\\
\mathbf{v}=0 & \text{on}\ \partial\left(\Omega\setminus S_{L_1}\right).
\end{cases}\]
As the compatibility condition 
\[\int_{\partial\left(\Omega\setminus S_{L_1}\right)}\ensuremath{\overline{\mathcal{F}}_{\delta}}\cdot\nu dA=\int_{\left\{ s=-L_1\right\} }\ensuremath{\overline{\mathcal{F}}_{\delta}}\cdot\nu dA=0\]
is fulfilled  because $0=\int_{\partial S_{L_1}}\mathcal{F}_{\delta}\cdot\nu_{\delta}dA=\int_{\left\{ s=-L\right\} }\mathcal{F}_{\delta}\cdot\nu_{\delta}dA+\int_{\left\{ s=L\right\} }\mathcal{F}_{\delta}\cdot\nu_{\delta}dA$ and
\[0=\int_{\partial\left(\Omega\cap S_{L_1}\right)}\mathcal{F}_{\delta}\cdot\nu_{\delta}dA=\int_{\partial\Omega}\mathcal{F}_{\delta}\cdot\nu_{\delta}dA+\int_{\left\{ s=L_1\right\} }\mathcal{F}_{\delta}\cdot\nu_{\delta}dA=\int_{\left\{ s=L\right\} }\mathcal{F}_{\delta}\cdot\nu_{\delta}dA
\]
is satisfied. Thus, we can set 
\[\mathcal{F}_{\delta}\left(\xi\right):=\overline{\mathcal{F}}_{\delta}\left(\xi\right)-\mathbf{v}.\]
We claim that this operator does satisfy all the assumed properties. Indeed, only the required regularity estimates are left to prove and  follow by direct computations as  performed in  details in the proof of Proposition~3.3 in~\cite{MS22}.
\end{proof}

 \section{The decoupled problem}\label{sec:decoupled-problem}
This section is dedicated to {\em formally} show how a-priori estimates can be derived for coupled solutions. All analysis performed here will then be rigorously be performed on the Galerkin level in the next section.

\subsection{Formal derivation of decoupled solutions}
 We start by introducing a suitable way of decoupled solutions. For that let $\delta$ be a given,  \emph{time-periodic} prescribed displacement $\ensuremath{\delta\in L^{\infty}\left(I;W^{1,2}\left(\omega\right)\right)\cap W^{1,\infty}\left(I;L^{2}\left(\omega\right)\right)}$.
Let us now assume  that  the equations~ \eqref{eqn:system} take place in $\Omega_{\delta}$.
We consider a pair of test functions $\left({\bf q},\xi\right)\in V_{T}^{\delta}$.
%
% Now, if 
%we  proceed exactly as in the derivation of \eqref{eqn:weak-form} we obtain (due to the decoupling) that 
%\begin{equation}
%\begin{aligned}
%&\frac{d}{dt}\int_{\Omega_{\delta\left(t\right)}}{\bf u}\cdot{\bf q}dx-\int_{\Gamma_{\delta\left(t\right)}}{\bf u}\cdot{\bf q}\left(\partial_{t}\delta\nu\right)\circ\psi_{\delta}^{-1}+\int_{\Omega_{\delta\left(t\right)}}-{\bf u}\cdot\partial_{t}{\bf q}dx
%\\
%&\quad + \int_{\Omega_{\delta\left(t\right)}}{\bf u}\otimes{\bf u}:\nabla{\bf q}dx+\nabla{\bf u}:\nabla{\bf q}dx+\int_{\partial\Omega_{\delta\left(t\right)}}\left|{\bf u}\right|^{2}{\bf q}\cdot\nu_{\delta}dA 
%\\
%&\quad +
%\frac{d}{dt}\int_{\omega}\partial_{t}\eta\xi dA-\int_{\omega}\partial_{t}\eta\partial_{t}\xi dA+K\left(\eta,\xi\right)+\int_{\Gamma_{p}}P{\bf q}\cdot\nu dA 
%\\
%&=
%\int_{\Omega_{\delta\left(t\right)}}{\bf f}\cdot{\bf q}dx+\int_{\omega}g\xi dA.
%\end{aligned}
%\end{equation}
In order to obtain a good energy balance for the decoupled system we may rewrite the time-derivative as
\begin{align}
\begin{aligned}
\label{eq:time}
\int_{\Omega_{\delta\left(t\right)}}\partial_t{\bf u}\cdot{\bf q}dx
&=\frac{d}{dt}\int_{\Omega_{\delta\left(t\right)}}\frac{{\bf u}\cdot{\bf q}}{2}dx+\int_{\Omega_{\delta\left(t\right)}}\frac{\partial_{t}{\bf u}\cdot{\bf q}-{\bf u}\cdot\partial_{t}{\bf q}}{2}dx-\int_{\Gamma_{\delta\left(t\right)}}\frac{{\bf u}\cdot{\bf q}}{2}\left(\partial_{t}\delta\nu\right)\circ\psi_{\delta\left(t\right)}^{-1}dA
\\
&=\frac{d}{dt}\int_{\Omega_{\delta\left(t\right)}}{\bf u}\cdot{\bf q}dx-\int_{\Omega_{\delta\left(t\right)}}{\bf u}\cdot\partial_{t}{\bf q}dx
-\int_{\Gamma_{\delta\left(t\right)}}{\bf u}\cdot{\bf q}\left(\partial_{t}\delta\nu\right)\circ\psi_{\delta\left(t\right)}^{-1}dA
\end{aligned}
\end{align}
The convective term can be rewritten as
\begin{align}
\begin{aligned}
\label{eq:conv}
\int_{\Omega_{\delta\left(t\right)}}\left({\bf u}\cdot\nabla\right){\bf u}\cdot{\bf q}dx
&=\int_{\partial\Omega_{\delta\left(t\right)}}\frac{\left|{\bf u}\right|^{2}{\bf q}\cdot\nu}{2}dA+\int_{\Omega_{\delta\left(t\right)}}\frac{\left({\bf u}\cdot\nabla\right){\bf u}\cdot{\bf q}-\left({\bf u}\cdot\nabla\right){\bf q}\cdot{\bf u}}{2}dx
\end{aligned}
\end{align}
In order to obtain a good coupled solution {\em that does satisfy the energy}, we make the following choice for the {\em weak formulation of the decoupled problem}, where we use the (coupled) boundary conditions. Indeed we will will seek a coupled solution $\left(\mathbf{u},\eta\right)$, that satisfies in an appropriate sense 
\begin{align}
\label{eq:forG}
\begin{aligned}
&\frac{d}{dt}\int_{\Omega_{\delta\left(t\right)}}\frac{{\bf u}\cdot{\bf q}}{2}dx+\int_{\Omega_{\delta\left(t\right)}}\frac{\partial_{t}{\bf u}\cdot{\bf q}-{\bf u}\cdot\partial_{t}{\bf q}}{2}+\frac{\left({\bf u}\cdot\nabla\right){\bf u}\cdot{\bf q}-\left({\bf u}\cdot\nabla\right){\bf q}\cdot{\bf u}}{2}dx
\\
&=\frac{d}{dt}\int_{\Omega_{\delta\left(t\right)}}{\bf u}\cdot{\bf q}dx +\int_{0}^{T}\int_{\Omega_{\delta\left(t\right)}}-{\bf u}\cdot\partial_{t}{\bf q}+\frac{\left({\bf u}\cdot\nabla\right){\bf u}\cdot{\bf q}-\left({\bf u}\cdot\nabla\right){\bf q}\cdot{\bf u}}{2}\, dx-\int_{\omega}\frac{1}{2}\partial_{t}\eta\partial_t\delta\xi J_{\delta}dA
\\
&=-\int_{\Omega_\delta}\nabla{\bf u}:\nabla{\bf q}dx
- \int_{\Gamma_{p}}P{\bf q}\cdot\nu\ dAdt-\int_{\omega}\partial_{t}^2\eta\xi dA-K\left(\eta,\xi\right) 
+\int_{\Omega_{\delta\left(t\right)}}{\bf f}\cdot{\bf q}dx+\int_{\omega}g\xi dA,
\end{aligned}
\end{align}
for all smooth coupled test-functions.

By using the test function $\left({\bf u},\partial_{t}\eta\right)$, applying  Reynolds' transport theorem we obtain (formally)
\begin{equation}\label{eqn:formal-energy-balance}
\frac{d}{dt}E\left(t\right)+\int_{\Omega_{\delta\left(t\right)}}\left|\nabla\mathbf{u}\right|^{2}dx+\int_{\Gamma_{p}}\mathbf{u}\cdot\nu P dA=\int_{\Omega_{\delta\left(t\right)}}\mathbf{f}\cdot\mathbf{u}dx+\int_{M}g\partial_{t}\eta dA
\end{equation}
where we have denoted the energy by
\begin{equation}
E\left(t\right):=\frac{1}{2}\int_{\Omega_{\delta\left(t\right)}}\left|\mathbf{u}\left(t,\cdot\right)\right|^{2}dx+\frac{1}{2}\int_{\omega}\left|\partial_{t}\eta\left(t,\cdot\right)\right|^{2}dA+K\left(\eta\left(t,\cdot\right)\right),\ t\in I.
\end{equation}

This allows to introduce the following weak equation.

\begin{definition} 
\label{def:decoupled}
For $\delta:I\times \omega\to [-L,L]$ and ${\bf v}\in C^1([0,T]\times\omega)$ with $\divergence({\bf v})=0$  we call
\begin{align*}
(\bu,\eta)\in & L^{2}\left(I;H_{\text{div}}^{1}\left(\Omega_{\delta}\right)\right)\cap L^{\infty}\left(I;L^{2}\left(\Omega_{\delta}\right)\right)
\times
L^{\infty}\left(I;H_{0}^{2}\left(\omega\right)\right)\cap W^{1,\infty}\left(I;L^{2}\left(\omega\right)\right),
\end{align*}
such that $\mathbf{u}=0\text{ on }\Gamma_{D}, \  \mathbf{u}\cdot\tau_{1}=\mathbf{u}\cdot\tau_{2}=0\text{ on }\Gamma_{p},\mathbf{u}\circ\phi_{\delta}=\partial_{t}\eta\nu\text{ on }\omega$ a {\em weak solution on $[0,T]$} to the {\em linearized decoupled problem}, if
\begin{equation}\label{eqn:decoupled1}
\begin{aligned}
&\int_{\Omega_{\delta(T)}}\bu(T)\cdot\bfq(T)\, dx-\int_{\Omega_{\delta(0)}}\bu(0)\cdot\bfq(0)\, dx+\int_\omega \partial_t\eta(T)\cdot\xi(T)-\partial_t\eta(0)\xi(0)\, dA
\\
&\quad +\int_{0}^{T}\int_{\Omega_{\delta\left(t\right)}}-{\bf u}\cdot\partial_{t}{\bf q}+\frac{\left({\bf v}\cdot\nabla\right){\bf u}\cdot{\bf q}-\left({\bf v}\cdot\nabla\right){\bf q}\cdot{\bf u}}{2}+\nabla{\bf u}:\nabla{\bf q}dxdt
\\
& \quad + \int_{0}^{T}\int_{\Gamma_{p}}P{\bf q}\cdot\nu\ dAdt-\int_{\omega}\frac{1}{2}\partial_{t}\eta\partial_t\delta\xi J_{\delta}dAdt-\int_{\omega}\partial_{t}\eta\partial_{t}\xi dAdt+K\left(\eta,\xi\right)dt 
\\
&=
\int_{0}^{T}\int_{\Omega_{\delta\left(t\right)}}{\bf f}\cdot{\bf q}dxdt+\int_{\omega}g\xi dAdt
\end{aligned}
\end{equation}
for all $
(\bfq,\xi)\in  H^{1}_0\left(I;H_{\text{div}}^{1}\left(\Omega_{\delta}\right)\right)
\times
L^2\left(I;H_{0}^{2}\left(\omega\right)\right)\cap W^{1,2}(I;L^2(\omega))
$,
such that $\mathbf{q}=0\text{ on }\Gamma_{D}, \  \mathbf{q}\cdot\tau_{1}=\mathbf{q}\cdot\tau_{2}=0\text{ on }\Gamma_{p},\mathbf{q}\circ\phi_{\delta}=\xi\nu\text{ on }\omega$
and that the energy inequality 
\begin{equation}\label{eqn:energy}
\frac{d}{dt}E\left(t\right)+\int_{\Omega_{\delta\left(t\right)}}\left|\nabla\mathbf{u}\right|^{2}dx+\int_{\Gamma_{p}}\mathbf{u}\cdot\nu P dA\leq \int_{\Omega_{\delta\left(t\right)}}\mathbf{f}\cdot\mathbf{u}dx+\int_{\omega}g\partial_{t}\eta dA
\end{equation}
 is satisfied distributionally in time. 

Let $\delta(0)=\delta(T)$. We call the pair $\left({\bf u},\eta\right)\in V_{S}^{\delta}$
%with $ \left\Vert \eta\right\Vert _{L_{t,x}^{\infty}}<L$
 a \emph{time-periodic weak solution} for the  \emph{decoupled problem}  provided that it holds 
 \begin{equation}\label{eqn:decoupled}
\begin{aligned}
&\int_{0}^{T}\int_{\Omega_{\delta\left(t\right)}}-{\bf u}\cdot\partial_{t}{\bf q}+\frac{\left({\bf u}\cdot\nabla\right){\bf u}\cdot{\bf q}-\left({\bf u}\cdot\nabla\right){\bf q}\cdot{\bf u}}{2}+\nabla{\bf u}:\nabla{\bf q}dxdt
\\
& \quad + \int_{0}^{T}\int_{\Gamma_{p}}P{\bf q}\cdot\nu\ dAdt-\int_{\omega}\frac{1}{2}\partial_{t}\eta\partial_t\delta\xi J_{\delta}dAdt-\int_{\omega}\partial_{t}\eta\partial_{t}\xi dAdt+K\left(\eta,\xi\right)dt 
\\
&=
\int_{0}^{T}\int_{\Omega_{\delta\left(t\right)}}{\bf f}\cdot{\bf q}dxdt+\int_{\omega}g\xi dAdt,
\end{aligned}
\end{equation}
for all $(\bfq,\xi)\in V_{T}^{\delta}$ and provided \eqref{eqn:energy} is satisfied.
\end{definition}

\subsection{A-priori estimates} 
Here we present the key estimates, even so we will later perform the below estimate on the Galerkin level, actually weak solution do satisfy the estimates below. We show it here under conditions for a given geometry that is given by $\delta:[0,T]\times \Omega\to [-L,L]$. The dependence on this given function is however made so precise that the same estimate holds for $\eta\equiv\delta$, if the data is small enough.

First we show how the {\em energy estimate} can be used to get some first a-priori estimates, which do however {\em not suffice} to get existence of time-periodic solutions. 
\begin{lemma}
\label{lem:energy}
    There is a constant $c_M$ that depends on $\kappa,L$ and a constant $c$ that depends on  $\Gamma_D,\Gamma_P$, such that for any $\bu\in L^2(0,T;W^{1,2}(\Omega_\delta))$, with $\partial_t\eta(t,x)=\bu(t,\psi_{\delta}(t,x))$ on $\omega$ satisfying
    \begin{equation}\label{eqn:energy-weak}
\int_0^T\int_{\Omega_{\delta\left(t\right)}}\left|\nabla\mathbf{u}\right|^{2}dx+\int_{\Gamma_{p}}\mathbf{u}\cdot\nu P dA\, dt\leq \int_0^T\int_{\Omega_{\delta\left(t\right)}}\mathbf{f}\cdot\mathbf{u}dx+\int_{\omega}g\partial_{t}\eta dA\, dt,
\end{equation}
  then we find 
    \begin{equation}
    \norm{\partial_t\eta}_{L^2_tL_x^2)}+\left\Vert \mathbf{u}\right\Vert _{L_{t}^{2}W_{x}^{1,2}}^{2}\leq c\int_{0}^{T}\int_{\mathbb{R}^{3}}\left|\mathbf{f}\right|^{2}dxdr+c_M\int_{0}^{T}\int_{\omega}\left|g\right|^{2}dAdt+c\int_{0}^{T}\int_{\Gamma_p}P^{2}dAdt.
\end{equation}
\end{lemma}
\begin{proof}
%    We integrate \eqref{eqn:formal-energy-balance} from $0$ to $T$ and obtain by the periodicity of $\eta,\mathbf{u}$ that
%\begin{equation}
%\int_{0}^{T}\int_{\Omega_{\delta\left(t\right)}}\left|\nabla\mathbf{u}\right|^{2}dxdt+\int_{\Gamma_{p}}\mathbf{u}\cdot\nu P\left(t\right)dAdt=\int_{0}^{T}\int_{\Omega_{\delta\left(t\right)}}\mathbf{f}\cdot\mathbf{u}dxdt+\int_{0}^{T}\int_{\omega}g\partial_{t}\eta dAdt.
%\end{equation}
First note that $\mathbf{u}$ vanishes on $\Gamma_D$, which has positive area. Hence we can use Poincar\'{e}'s inequality.
\[
\norm{\bu}_{L^2(\Omega_\delta)}\leq c\norm{\bu}_{W^{1,2}(\Omega_\delta)}
\]
Next observe, that by Lemma \ref{lem:trace}
\[
\norm{\partial_t\eta}_{L^2(\omega)}^2\leq c(\kappa)\norm{\partial_t\eta\circ\phi^{-1}}_{L^2(M)}^2\leq c(\kappa,L) \norm{\bu}_{W^{1,2}(\Omega_\delta)}^2\]
 Further the trace of $\Gamma_P$ can be estimated using the assumption on the smoothness of $\Gamma_p$. Hence we reach at the following \emph{diffusion estimate}
\begin{equation}\label{eqn:difusion-est}
\int_{0}^{T}\norm{\bu}_{W^{1,2}(\Omega_\delta)}^2dt\leq c \int_{0}^{T}\int_{\mathbb{R}^{3}}\left|\mathbf{f}\right|^{2}dxdr+c(\kappa,L)\int_{0}^{T}\int_{\omega}\left|g\right|^{2}dAdt+c\int_{0}^{T}\int_{\Gamma_p}P^{2}dAdt
\end{equation}
and the estimate is shown. 
\end{proof}
Now we proceed with the extra regularity estimate. For that fix $\alpha>0$, such that
\begin{align}
\label{eq:alpha}
\alpha \norm{\eta}_{W^{1,2}_x}^2\leq K(\eta).
\end{align}
Please note, that $\alpha$ can be rather large, if the Lam\'e constants are large.

\begin{proposition}\label{prop:formal-decoupled-estimates}
Consider a positive constant $\tilde{M}$ and a function $\delta:[0,T]\times \omega\to \mathbb{R}$, such that
\begin{equation}\label{eqn:bound-M-delta}
\alpha\left\Vert \delta\right\Vert _{L_{t}^{\infty}W_{x}^{1,2}}^2+\frac{1}{2}\left\Vert \partial_{t}\delta\right\Vert _{L_{t}^{\infty}L_{x}^{2}}^2\leq \tilde{M}^2.
\end{equation} 
Further assume that 
\begin{equation}\label{eqn:bound-L-delta}
    \left\Vert \delta\right\Vert _{L_{t,x}^{\infty}}\le L
\end{equation}
and
\begin{align}
\label{cond1}
 \int_{0}^{T}\int_{\mathbb{R}^{3}}\left|\mathbf{f}\right|^{2}dxdr+c(\kappa,L)\int_{0}^{T}\int_{\omega}\left|g\right|^{2}dAdt+\int_{0}^{T}\int_{\Gamma_p}P^{2}dAdt\leq \tilde{C}.
\end{align}
Then the following holds for the time-periodic weak solution of the decoupled problem according to Definition~\ref{def:decoupled}:
\begin{enumerate}[(a)]
%where $$ will denote from now on
%\begin{equation}
%C\left({\bf f},g,P\right):=\int_{0}^{T}\int_{\mathbb{R}^{3}}\left|\mathbf{f}\right|^{2}dxdr+\int_{0}^{T}\int_{\omega}\left|g\right|^{2}dAdt+\int_{0}^{T}\int_{\Gamma_p}P^{2}dAdt.
%\end{equation}
\item There are constants $c$ depending only on $\Gamma_D, \Gamma_P$ and $c_M$ depending on $\kappa$ and $L$, such that for any $\theta\in (0,1)$
\begin{equation}
\label{eq:est1}
    \sup_{t\in I}E\left(t\right)\le \frac{c}{\theta(T-\theta)}\frac{1}{\theta}\left(\tilde{C}^{2}c_M^2+\tilde{C}c_M\right)+ \frac{T^2+T+1}{T(T-\theta)}\tilde{C} + \frac{\theta}{T-\theta} \tilde{M}^2,
\end{equation}
\item If $\tilde{C}$ satisfies 
\begin{align}
\label{cond2}
\frac{c}{\theta(1-\theta)}\frac{1}{\theta}\left(\tilde{C}^{2}c_M^2+\tilde{C}c_M\right) + \frac{T^2+T+1}{T(1-\theta)}\tilde{C}=\tilde{M}^2,
\end{align}
for some $\theta\in (0,1)$, then 
\[
\sup_{t\in I}E\left(t\right)\le \tilde{M}^2.
\]
\end{enumerate}

%
%\begin{remark} We would like to point out that the conditions \eqref{eqn:bound-M-delta} and \eqref{eqn:bound-L-delta} are ensured if one assumes that the quantity 
%\begin{equation}
%\tilde{M}:=\left\Vert \delta\right\Vert _{L_{t}^{\infty}W_{x}^{2,2}}+\left\Vert \partial_{t}\delta\right\Vert _{L_{t}^{2}L_{x}^{2}}
%\end{equation} is smaller that a certain constant $C(\Omega,L)$ so that  $\tilde{M}<C$ implies \eqref{eqn:bound-L-delta}. This improves the right-hand side of \eqref{eqn:sup-E-with-M} to 
%\begin{equation}
%\sup_{t\in I}E\left(t\right)\le C\left(\Omega,L,T\right)\cdot\left(C\left(\mathbf{f},g,P\right)^{2}+C\left(\mathbf{f},g,P\right)\right)+C\left(\mathbf{f},g,P\right)M.
%\end{equation}
%
%
%On another hand,  if  $\delta = \lambda \eta$ for some $\lambda\in\left[0,1\right]$, then $M\lesssim\left\Vert \partial_{t}\eta\right\Vert _{L_{t}^{2}L_{x}^{2}}+\left(\sup_{t\in I}E\left(t\right)\right)^{1/2}$ and, \emph{provided that $C\left(\mathbf{f},g,P\right)$ is sufficiently small},
%\eqref{eqn:sup-E-with-M} becomes:
%
%\begin{equation}\label{eqn:sup-E-t-uniform-o(sigma)}
%\sup_{t\in I}E\left(t\right)\le C.
%\end{equation}
%\end{remark}
\end{proposition}

\begin{proof}
Note that Lemma \ref{lem:energy} implies that
\begin{equation}
   \norm{\partial_t\eta}_{L^2_tL_x^2} ^{2} + \left\Vert \mathbf{u}\right\Vert _{L_{t}^{2}W_{x}^{1,2}}^{2}\leq  %C\left({\bf f},g,P\right)
   c\tilde{C}.
\end{equation}
From the mean value theorem there is $t_0 \in [0,T]$ for which
\[E\left(t_{0}\right)=\fint_{0}^{T}E\left(t\right)dt\]
and  by integrating the equation~\eqref{eqn:formal-energy-balance} from $t_0$ to $t$ and using \eqref{eqn:difusion-est}  we obtain
\begin{equation}\label{eqn:mean-K}
\sup_{0\le t\le T}E\left(t\right)\lesssim\fint_{0}^{T}E\left(t\right)dt+\tilde{C}.
\end{equation}
But the diffusion estimate~\eqref{eqn:difusion-est} ensures
\begin{equation}
\begin{aligned}\fint_{0}^{T}E\left(t\right)dt & \le\fint_{0}^{T}\int_{\Omega_{\delta\left(t\right)}}\left|\mathbf{u}\left(t,\cdot\right)\right|^{2}dxdt+\int_{\omega}\left|\partial_{t}\eta\left(t,\cdot\right)\right|^{2}dAdt+K\left(\eta\left(t,\cdot\right)\right)dt\\
 & \lesssim\frac{1}{T}\tilde{C}+\fint_{0}^{T}K\left(\eta\left(t,\cdot\right)\right)dt.
\end{aligned}
\end{equation}
and \eqref{eqn:mean-K} reads now
\begin{equation}\label{eqn:sup-t-E-part1}
\sup_{0\le t\lesssim T}E\left(t\right)\le\Big(1+\frac{1}{T}\Big)\tilde{C}+\fint_{0}^{T}K\left(\eta\left(t,\cdot\right)\right)
\end{equation}
In order to estimate the Koiter  term $\fint_{0}^{T}K\left(\eta\left(t,\cdot\right)\right)$ 
 we  consider  the test-function
  $\left(\mathcal{F}_{\delta}\left(\eta\right),\eta\right)\in V_{T}^{\delta}$. 
In the following we will show that for every $\theta\in (0,1)$ there is some constant $c\left(\theta,\Omega,L,T,\mathbf{f},g,P\right)$ such that
\begin{equation}\label{eqn:claim-est-K}
\int_{0}^{T}K\left(\eta\left(t\right)\right)dt\le \theta\sup_{0\le t\le T}E\left(t\right)+c\left(\Omega,L,T,\mathbf{f},g,P\right)+\theta M^2.
\end{equation}
Indeed, by  testing with   $\left(\mathcal{F}_{\delta}\left(\eta\right),\eta\right)\in V_{T}^{\delta}$ in \eqref{eqn:decoupled} we obtain
\begin{equation}
\begin{aligned}2\int_{0}^{T}K\left(\eta\left(t\right)\right)dt\le & \int_{0}^{T}\int_{\omega}\left|\partial_{t}\eta\right|^{2}dAdt+\int_{0}^{T}\int_{\omega}g\eta dAdt%+\frac{1}{2}\int_{\omega}\left|\partial_{t}\eta\right|^{2}\eta dAdt
+\int_{0}^{T}\int_{\Omega_{\delta\left(t\right)}}\mathbf{u}\cdot\partial_{t}\mathcal{F_{\delta}}\left(\eta\right)dxdt 
\\
 & 
 - \int_{0}^{T}\int_{\Omega_{\delta\left(t\right)}} \nabla\mathbf{u}:\nabla\mathcal{F_{\delta}}\left(\eta\right)\, dxdt
 +\int_{0}^{T}\int_{\Omega_{\delta\left(t\right)}}\left(\mathbf{u}\cdot\nabla\right)\mathcal{F_{\delta}}\left(\eta\right)\cdot\mathbf{u}dxdt\\
 & +\int_{0}^{T}\int_{\Omega_{\delta\left(t\right)}}\mathbf{f}\cdot\mathcal{F_{\delta}}\left(\eta\right)dx+\int_{\Gamma_{p}}P\mathcal{F_{\delta}}\left(\eta\right)\cdot\nu\,dAdt
  =:\sum_{k=1}^{7}I_{k}
\end{aligned}
\end{equation}
Now, for an arbitrary $\theta>0$ we obtain using H\"older's inequality and \eqref{eq:alpha} that
\begin{equation}\label{eqn:I1,I2,I3}
    \begin{aligned}\begin{aligned}\left|I_{1}\right|+\left|I_{2}\right|\end{aligned}
 & \lesssim \tilde{C}+\left\Vert g\right\Vert _{L_{t}^{2}L_{x}^{2}}\left\Vert \eta\right\Vert _{L_{t}^{2}L_{x}^{2}}%+\left\Vert \eta\right\Vert _{L_{t,x}^{\infty}}C\left(\mathbf{f},g,P\right)
 \\
 & \lesssim \tilde{C}+\tilde{C}^{1/2}\left(\sup_{t\in I}E\left(t\right)\right)^{1/2}%+\left(\sup_{t\in I}E\left(t\right)\right)^{1/2}C\left(\mathbf{f},g,P\right)\lesssim\theta\left(\sup_{t\in I}E\left(t\right)\right)+\frac{4}{\theta}C\left(\mathbf{f},g,P\right)
\end{aligned}
\end{equation}
Recalling the properties of the operator $\mathcal{F}_{\delta}$ (see Proposition~\ref{prop:Solenoidal-extension-o}) we can estimate the further terms as follows:
\begin{equation}\label{eqn:I8,I9}
    \begin{aligned}\left|I_{6}\right|+\left|I_{7}\right| 
    & \lesssim
    \left\Vert \mathbf{f}\right\Vert _{L_{t}^{2}L_{x}^{2}}\left\Vert \mathcal{F}_{\delta}\left(\eta\right)\right\Vert _{L_{t}^{2}L_{x}^{2}}+\left\Vert P\right\Vert _{L_{t}^{2}L_{x}^{2}}\left\Vert \mathcal{F}_{\delta}\left(\eta\right)\right\Vert _{L_{t}^{2}W_{x}^{1,2}}
    \\
 & \lesssim (\tilde{C})^{1/2}\left\Vert \mathcal{F}_{\delta}\left(\eta\right)\right\Vert _{L_{t}^{2}W_{x}^{1,2}}\\
 & \lesssim\tilde{C}^{1/2}\left(\left\Vert \eta\right\Vert _{L_{t}^{2}W_{x}^{1,2}}+\left\Vert \eta\nabla\delta\right\Vert _{L_{t}^{2}L_{x}^{2}}\right)\\
 & \lesssim\tilde{C}^{1/2} c(L,\kappa)\Big(\left(\sup_{t\in I}E\left(t\right)\right)^{1/2}+\tilde{M}\Big)\\
 & \lesssim\theta\left(\sup_{t\in I}E\left(t\right)+M^2\right)+\frac{4}{\theta}\left(\tilde{C}+\tilde{C}^2\right)c^2(L,\kappa)
\end{aligned}
\end{equation}
and then the most tedious parts as follows:
\begin{equation}\label{eqn:I4,I5,I6,I7}
\begin{aligned}\left|I_{3}\right|+\left|I_{4}\right|+\left|I_{5}\right|
\le & \left\Vert \mathbf{u}\right\Vert _{L_{t}^{2}L_{x}^{2}}\left\Vert \partial_{t}\mathcal{F}_{\delta}\left(\eta\right)\right\Vert _{L_{t}^{2}L_{x}^{2}}+\left\Vert \mathbf{u}\right\Vert _{L_{t}^{2}W_{x}^{1,2}}\left\Vert \mathcal{F}_{\delta}\left(\eta\right)\right\Vert _{L_{t}^{2}W_{x}^{1,2}}
\\
 &+ \left\Vert \left|\mathbf{u}\right|^{2}\right\Vert _{L_{t}^{1}L_{x}^{3}}\left\Vert \mathcal{F}_{\delta}\left(\eta\right)\right\Vert _{L_{t}^{\infty}W_{x}^{1,3/2}}%+\left\Vert \mathbf{u}\right\Vert _{L_{t}^{2}W_{x}^{1,2}}^{2}\left\Vert \mathcal{F}_{\delta}\left(\eta\right)\right\Vert _{L_{t}^{\infty}L_{x}^{3}}
 \\
\lesssim & \tilde{C}^{1/2}\left(\left\Vert \partial_{t}\mathcal{F}_{\delta}\left(\eta\right)\right\Vert _{L_{t}^{2}L_{x}^{2}}
+\left\Vert \mathcal{F}_{\delta}\left(\eta\right)\right\Vert _{L_{t}^{\infty}W_{x}^{1,2}}\right)
\\
& + \tilde{C}\left\Vert \mathcal{F}_{\delta}\left(\eta\right)\right\Vert _{L_{t}^{\infty}W_{x}^{1,3/2}}
 %\\
%\lesssim & \underbrace{\left(\tilde{C}^{1/2}+\left(C\left(\mathbf{f},g,P\right)\right)\right)}_{\sigma}\cdot\left(\left\Vert \partial_{t}\mathcal{F}_{\delta}\left(\eta\right)\right\Vert _{L_{t}^{2}L_{x}^{2}}+\left\Vert \mathcal{F}_{\delta}\left(\eta\right)\right\Vert _{L_{t}^{\infty}W_{x}^{1,2}}\right)
\\
\lesssim & \tilde{C}^{1/2}c(L,\kappa)\left(\left\Vert \partial_{t}\eta\right\Vert _{L_{t}^{2}L_{x}^{2}}+\left\Vert \eta\partial_{t}\delta\right\Vert _{L_{t}^{2}L_{x}^{2}}+\left\Vert \eta\right\Vert _{L_{t}^{\infty}W_{x}^{1,2}}+\left\Vert \eta\nabla\delta\right\Vert _{L_{t}^{\infty}L_{x}^{2}}\right)
\\
&\quad +\tilde{C}c(L,\kappa_1,\kappa_2)\Big(\left\Vert \eta\right\Vert _{L_{t}^{\infty}W_{x}^{1,3/2}}+\left\Vert \eta\nabla\delta\right\Vert _{L_{t}^{\infty}L_{x}^{3/2}}\Big)
\\
\lesssim & \frac{1}{\theta}\left(\tilde{C}^{2}c(L,\kappa)^2+\tilde{C}c(L,\kappa_1,\kappa_2)\right)
+\theta \left(\sup_{t\in I}E\left(t\right)+\tilde{M}^2\right)
\end{aligned}
\end{equation}
Recall that  $\tilde{M}$ was introduced in \eqref{eqn:bound-M-delta}. By combining \eqref{eqn:I1,I2,I3}, \eqref{eqn:I8,I9}  we obtain that 
\begin{equation}
\label{eq:R1}
\int_{0}^{T}K\left(\eta\right)dt\le \frac{1}{\theta}\left(\tilde{C}^{2}c(L,\kappa)^2+\tilde{C}c(L,\kappa_1,\kappa_2)\right)
+\theta \left(\sup_{t\in I}E\left(t\right)+\tilde{M}^2\right)+\tilde{C}
\end{equation}
which proves the claim~\eqref{eqn:claim-est-K}.
Now \eqref{eqn:sup-t-E-part1}implies, for a sufficiently small $\theta>0$ and $C\left(\mathbf{f},g,P\right)\leq \tilde{C}$ as follows
\begin{equation}\label{eqn:R}
\sup_{t\in I}E\left(t\right)
\le \frac{1}{T\theta}\left(
\tilde{C}^2+\tilde{C}\right)c(L,\kappa)^2+\frac{\theta}{T}\left(\sup_{t\in I}E\left(t\right)+\tilde{M}^2\right)
+\frac{T^2+T+1}{T^2}\tilde{C}.
\end{equation}
%If, moreover, $\delta=\lambda \eta$ for some $\lambda \in [0,1]$, then 
%\[
%M=\lambda\left(\left\Vert \partial_{t}\eta\right\Vert _{L_{t}^{2}L_{x}^{2}}+\left\Vert \eta\right\Vert _{L_{t}^{\infty}W_{x}^{1,2}}\right)\le\left(C\left(\mathbf{f},g,P\right)\right)^{1/2}+\left(\sup_{t\in I}E\left(t\right)\right)^{1/2}
%\]
%
% $C\left(\mathbf{f},g,P\right)$ it holds that 
Which implies \eqref{eq:est1}
which finishes (a). Finally (b) just follows by calculations.
\end{proof}
The next corollary shows how to close the circle and get a uniform estimate, provided the given data is small enough.
\begin{corollary}
\label{cor}
There exist $\tilde{M}>0$, such that for all $\eta:[0,T]\times \omega\to \mathbb{R}$ be given, such that $\text{ess sup}_{t\in [0,T]}K(\eta(t))\leq \tilde{M}^2$, we find $\norm{\eta}_\infty\leq L$.

Further, if $\tilde{C}$ satisfies \eqref{cond1}, $\delta$ satisfies \eqref{eqn:bound-M-delta} and \eqref{eqn:bound-L-delta} and
 $f,g,P$ satisfy \eqref{cond2} we find for the weak solution $(\eta,{\bf u})$ of Definition~\ref{def:decoupled} that by \eqref{eq:alpha}
\[
\alpha\left\Vert \eta\right\Vert _{L_{t}^{\infty}W_{x}^{1,2}}^2+\frac{1}{2}\left\Vert \partial_{t}\eta\right\Vert _{L_{t}^{2}L_{x}^{2}}^2\leq \text{ess sup}_{t\in [0,T]} E\left(t\right)\le \tilde{M}^2,
\]
which implies
\[
\text{ess sup}_{t\in [0,T]}K(\eta(t))\leq \tilde{M}^2\text{ and } \norm{\eta}_\infty\leq L.
\]
\end{corollary}
\begin{proof}
By Sobolev embedding we find using the zero trace of $\eta$ in $\partial \omega$ that for a constant $c$ depending on $\omega$ 
\[
\norm{\eta(t)}_{L^\infty_x}^2\leq c \norm{\eta(t)}_{H^2_x}^2\leq \frac{c}{c_0} K(\eta(t))\leq L,
\]
if $\tilde{M}$ is small enough.
This implies the first claim. The second claim now follows directly from Proposition~\ref{prop:formal-decoupled-estimates}.
\end{proof}
Finally we need the following statement for the case that the solution is {\em not time-periodic}
\begin{corollary}
\label{cor:gal}
Assume that \eqref{eqn:bound-M-delta}--\eqref{cond1} are satisfied. Let $(\bu,\eta)$ be a {\em weak solution on $[0,T]$} in the sense of Definition~\ref{def:decoupled}, satisfying \eqref{eqn:energy-weak}, then
\[
\sup_{t\in I}E\left(t\right)+\norm{\partial_t\eta}_{L^2_tL_x^2} ^{2}+\left\Vert \mathbf{u}\right\Vert _{L_{t}^{2}W_{x}^{1,2}}^{2}\leq \hat{C},
\]
with $\hat{C}$ depending on $\kappa,L,f,g,P,\Gamma_p,\Gamma_D,\abs{M},\tilde{M},T,{\bf v}$.
\end{corollary}
\begin{proof}
First the dissipation terms are directly estimated by Lemma~\ref{lem:energy}. In order to estimate the energy $E$ we follow Proposition~\ref{prop:formal-decoupled-estimates} with respective dependencies of ${\bf v}$. The main difference is that when taking $(\left(\mathcal{F}_{\delta}\left(\eta\right),\eta\right)$ as a test function, boundary terms do appear. This is circumvented by integrating from $[a,T-a]$ where $a\in [0,\frac{T}{4}]$. Applying the estimates of Proposition~\ref{prop:Solenoidal-extension-o}, namely up to \eqref{eq:R1} implies
\begin{align*}
\int_{a}^{T-a}K\left(\eta\right)dt
&\le C
+\theta \sup_{t\in I}E\left(t\right)
\\
&\quad +\int_{\Omega_{\delta(a)}}\bu(a)\cdot\Fcal_{\delta}(\eta)(a)\, dx-\int_{\Omega_{\delta(T-a)}}\bu(T-a)\cdot\Fcal_{\delta}(\eta)(T-A)\, dx\\
&\quad +\int_\omega \partial_t\eta(a)\cdot\eta(a)-\partial_t\eta(T-a)\eta(T-a)\, dA,
\end{align*}
where here and in the following we use $C$ as a constant depending on $\kappa,L,f,g,P,\Gamma_p,\Gamma_D,\abs{M},\tilde{M},T$.
Now we integrate $a$ over $[0,\frac{T}{4}]$, to find
\begin{align*}
\int_{T/4}^{3T/4}K\left(\eta\right)dt
&\le \fint_0^{T/4}\int_{a}^{T-a}K\left(\eta\right)dt\,da
\le C+\theta\sup_{t\in I}E(t)
 +\fint_0^{T/4}\int_{\Omega_{\delta(a)}}\bu(a)\cdot\Fcal_{\delta}(\eta)(a)\, dx
 \\
 &\quad -\int_{\Omega_{\delta(T-a)}}\bu(T-a)\cdot\Fcal_{\delta}(\eta)(T-A)\, dx\, da
 \\
 &\quad +\fint_0^{T/4}\int_\omega \partial_t\eta(a)\cdot\eta(a)-\partial_t\eta(T-a)\eta(T-a)\, dA\, da,
\end{align*}
Now Young's inequality, the dissipation estimate and the fact, that $\norm{\eta}_{L^2_x}+\norm{\Fcal_\delta(\eta)}_{L^2_x}\leq c(L,\kappa)K(eta)$ 
implies that for any $\theta\in (0,1)$ there is a constant $C$, such that
\[
\fint_{T/4}^{3T/4}K\left(\eta\right)dt
\le \frac{2}{T}\fint_0^{T/4}\int_{a}^{T-a}K\left(\eta\right)dt\,da
\le C+\theta \sup_{t\in I}E\left(t\right).
\]
But now by the mean value theorem, there is a $t_0\in[\frac{T}{4},\frac{3T}{4}]$, such that $E(t_0)=\fint_{T/4}^{3T/4}E\left(\eta\right)dt$. But now \eqref{eqn:energy} and the diffusive estimate implies that there is a $C$ such that
\[
\sup_{t\in I}E\left(t\right)\leq E(t_0)+C\leq C+\theta \sup_{t\in I}E\left(t\right),
\] 
which finishes the estimate for $\theta \le \frac{1}{2}$.
\end{proof}
\section{Proof of the main results}\label{sec:Proof-main-thm}
The existence of a solution follows,  by considering the limit of a Galerkin approximations. Note that here $\eta$ is both part of the domain $\Omega_{\eta}$ and of the solution of the shell equation, so it is necessary to decouple the problem, that is to consider domains $\Omega_{\delta}$ with $\delta$ satisfying the condition \eqref{eqn:bound-M-delta} as in Section~\ref{sec:decoupled-problem}.  
 The proof is split into several parts. In particular two fixed points are performed.
 First, the existence of a discrete time-periodic  solution $\left(\mathbf{u}_{n},\eta_{n}\right)$ of the system is showed, by a fixed-point procedure.
 In a second step, we employ a set-valued fixed-point argument of the form $\delta_{n}\mapsto\left\{ \eta_{n}:\eta_{n}\ \text{is a time-periodic weak solution }\right\} $ and  we recover the coupling. 
 The last  part consists in passing to the limit in the weak-formulation as $n\to\infty$.
Here, as common,  proving that $\left(\mathbf{u}_{n},\partial_{t}\eta_{n}\right)\to\left(\mathbf{u},\partial_{t}\eta\right)$ \emph{strongly} in $L^2$ is the main difficulty. In order to pass with the Galerkin basis {\em that depends on the solution} to the limit, we need to include a stabilization term. Hence we will show that the existence of a weak solution in the following sense.
\begin{definition}[Stabilized solution]\label{def:weak-soln-eps}
We call a pair $(\mathbf{u}, \ \eta)\in V^{\eta}_{S}$  with $\left\Vert \eta\right\Vert _{L_{t,x}^{\infty}}<L$ and $\partial_t\eta\in L^2_t H^3_x$ a \emph{stabilized time-periodic weak solution} provided that  the following relation holds:

\begin{equation}\label{eqn:epsilonlevel}
\begin{aligned}\int_{0}^{T}\int_{\Omega_{\eta\left(t\right)}}-\mathbf{u}\cdot\partial_{t}\mathbf{q}+\nabla\mathbf{u}:\nabla\mathbf{q}+\frac{1}{2}\left(\mathbf{u}\cdot\nabla\right)\mathbf{u}\cdot\mathbf{q}-\frac{1}{2}\left(\mathbf{u}\cdot\nabla\right)\mathbf{q}\cdot\mathbf{u}dxdt & +\\
\int_{0}^{T}\int_{\omega}\varepsilon\nabla^{3}\partial_{t}\eta\nabla^{3}\xi-\partial_{t}\eta\partial_{t}\xi dAdt+\int_{0}^{T}K\left(\eta,\xi\right)dt+\int_{0}^{T}\int_{\Gamma_{p}}P\mathbf{q}\cdot\nu\,dAdt & =\\
\int_{0}^{T}\int_{\Omega_{\eta(t)}}\mathbf{f}\cdot\mathbf{q}dxdt+\int_{0}^{T}\int_{\omega}g\xi dAdt\quad\text{ for all }\left(\mathbf{q},\xi\right)\in V_{T}^{\eta}
\end{aligned}
\end{equation}
and that the energy inequality 
\begin{equation*}%\label{eqn:energy}
\frac{d}{dt}E\left(t\right)+\int_{\Omega_{\eta\left(t\right)}}\left|\nabla\mathbf{u}\right|^{2}dx+\int_{\Gamma_{p}}\mathbf{u}\cdot\nu P dA\leq \int_{\Omega_{\eta\left(t\right)}}\mathbf{f}\cdot\mathbf{u}dx+\int_{\omega}g\partial_{t}\eta dA
\end{equation*}
 is satisfied distributionally in time.
\end{definition}
Accordingly we proof the following theorem, which then implies Theorem \ref{thm:main} by passing with $\varepsilon\to 0$. The construction theory can directly be adapted from \cite{MS22} or \cite{LR14}.

\begin{theorem}[Existence of a stabilized periodic solution]
\label{thm:epsilon}
    There exists a constant $\tilde{C}$ depending on $\Gamma_p,\Gamma_D,\kappa,L,\abs{M}$ and $c_0$ such that if
$\left(\mathbf{f},g,P\right)\in L_{\text{per}}^{2}\left(I;L^2(\mathbb{R}^{3})\right)\times L_{\text{per}}^{2}\left(I;L^2(\omega)\right)\times L_{\text{per}}^{2}\left(I;L^2(\Gamma_p)\right)$ satisfies
\begin{equation}
\left\Vert \mathbf{f}\right\Vert _{L_{t}^{2}L_{x}^{2}}+\left\Vert g\right\Vert _{L_{t}^{2}L_{x}^{2}}+\left\Vert P\right\Vert _{L_{t}^{2}L_{x}^{2}}\le\tilde{C}
\end{equation}
then there exists at least one weak time-periodic solution $\left(\mathbf{u},\eta\right)$ as it is defined in Definition~\ref{def:weak-soln-eps}. 
Furthermore, it enjoys the \emph{diffusion estimate}
\begin{align*}
\frac{1}{2}\int_{0}^{T}\int_{\Omega_{\eta\left(t\right)}}\left|\nabla\mathbf{u}\right|^{2}\,dx+\varepsilon\int_\omega \abs{\nabla^2\partial_t\eta}\, dA\, dt\leq\int_{0}^{T}\int_{\Omega_{\eta\left(t\right)}}\mathbf{f}\cdot{\bf u}\,dxdt+\int_{\omega}g\partial_{t}\eta\,dAdt+\int_{\Gamma_{p}}P{\bf u}\cdot\nu\,dAdt
\end{align*}
and the {\em additional regularity estimate}
\begin{equation*}
\sup_{t\in I}E\left(t\right)+\left\Vert \mathbf{u}\right\Vert _{L_{t}^{2}W_{x}^{1,2}}^{2}+\left\Vert \mathbf{\partial_t\eta}\right\Vert _{L_{t}^{2}W_{x}^{2,2}}^{2}\leq C,
\end{equation*}
with $C$ depending on $\tilde{C}$, $\Gamma_p,\Gamma_D,\kappa_1,\kappa_2,L,\abs{M},c_0$ and $T$,
where we denote again
\begin{equation}
    E\left(t\right):=\frac{1}{2}\int_{\Omega_{\eta\left(t\right)}}\left|\mathbf{u}\right|^{2}dx+\frac{1}{2}\int_{\omega}\left|\partial_{t}\eta\right|^{2}dA+K\left(\eta\left(t,\cdot\right)\right),\quad t\in I.
\end{equation}
\end{theorem}
The existence follows in several steps. 
\begin{enumerate}
    \item[Step 1]For the decoupled problem a Galerkin solution is constructed, for the {\em Cauchy problem}, with arbitrary initial data. Here we introduce a special basis that allows to perform the wished for a-priori estimates on the Galerkin level.
    \item[Step 2]A first fixed point is perforemd to optain a time-periodic solution to the decoupled problem.
    \item[Step 3]
    The additional regularity estimate for the decoupled problem is obtained relying on Section \ref{sec:decoupled-problem}.
    \item[Step 4] Another fixed point is producing the coupled discrete time-periodic solution.
    \item[Step 5] The Galerkin limit is taken, showing in particular the $L^2$ compactness of the fluid and solid velocity.
\end{enumerate}

\subsection{Existence of a discrete time-periodic solution}
\subsubsection*{Construction of the basis}
The construction of the basis is performed in two steps. First we extend a basis for the shell equation into the moving fluid domain. Second we construct a basis for the fluid with zero trace. As long as the geometry of the fluid domain is {\em decoupled} of the shell equation, this part of the basis is only effecting the fluid equation. In the following we use non-bold fold letters for the basis functions used in to discretize the shell equation and bold-fold capital letters as basis functions for the fluid velocity.

We start by introducing a suitable basis for the {\em stabilized shell equation}. We choose the eigenvalues of $(-\Delta)^3$, with Dirichlet boundary values on $\partial \omega$. They span the space
 \[
 H_{0}^{3}\left(\omega\right)=\text{span}\left\{ \xi_{l}:l\in\mathbb{Z},\ l\ge1\right\}, 
 \]
 over which we seek the shell deformation $\eta$.  The basis can assumed to be a subset of $C^\infty(\omega)$, as we assume that $\omega$ has a smooth boundary. For later use, we also remark, that the projection to the discrete basis converges uniformly in $H^2$. This means that for every $\varepsilon>0$, there exists an $n_0$, such that for all $\xi\in H^3_0(\omega)$, there are $e_l$, such that
 \begin{align}
 \label{eq:unif}
 \norm{\xi-\sum_{l=1}^ne_l\xi_l}_{H^2}\leq \varepsilon,
 \end{align}
 for all $n\geq n_0$.
  We extend this basis of the boundary values to the whole domain $\Omega_{\delta}$ (for a smooth $\delta:[0,T]\times \omega\to (L,L)$)  by using  the operator $\mathcal{F}_{\delta}$ constructed in Proposition~\ref{prop:Solenoidal-extension-o}.
We set
 \begin{equation}
 {\bf Y}_{k}\left(t,\cdot\right):=\mathcal{F}_{\delta\left(t\right)}\xi_{k},\ k\ge1.
 \end{equation}
Next, we consider the Hilbert space 
\[
\mathcal{H}:=\left\{ \varphi\in H^{1}\left(\Omega\right)\,:\,\divergence(\mathbf{\varphi})=0,\,\text{tr}_{\Gamma_p}\left(\mathbf{\varphi}\right)\cdot \tau_1=0=\text{tr}_{\Gamma_p}\left(\mathbf{\varphi}\right)\cdot \tau_2=\text{tr}_{\partial\Omega\setminus\Gamma_P}(\mathbf{\varphi})\right\} 
\]
and let $\left(\hat{\mathbf{Z}}_{k}\right)_{k\ge1}$ be a basis of it. 
 We construct such a basis on a steady domain using the eigenvalues of the respective Stokes operator, which are known by Hilbert space theory to be a basis of $\mathcal{H}$ (see for instance \cite[Theorem 7, Appendix E]{evans}). Please note that due to the mixed boundary conditions the basis of $\mathcal{H}$ may only be assumed to be in $H^1$, nevertheless, as it is based on a selfadjoint operator, we find also here for every $\varepsilon>0$,  an $n_0$, such that for all $f\in \mathcal{H}$, there are $e_l$, such that
 \begin{align}
 \label{eq:unif2}
 \norm{f-\sum_{l=1}^ne_l\mathbf{Z}_l}_{L^2}\leq \epsilon,
 \end{align}
 for all $n\geq n_0$.
 For $\delta\in C^1(I;C^\infty(\omega))$ we use the push forward by the Piola transform that conserves solidarity onto 
\[
\mathcal{H}_\delta=\left\{ \varphi\in H^{1}\left(\Omega_\delta\right)\,:\,\divergence(\mathbf{\varphi})=0,\,\text{tr}_{\Gamma_p}\left(\mathbf{\varphi}\right)\cdot \tau_1=0=\text{tr}_{\Gamma_p}\left(\mathbf{\varphi}\right)\cdot \tau_2=\text{tr}_{\partial\Omega_\delta\setminus\Gamma_P}(\mathbf{\varphi})\right\}, 
\]
which is defined as
\[
\mathcal{J}_{\delta\left(t\right)}:\mathcal{H}\to \mathcal{H}_\delta,\quad 
 \mathcal{J}_{\delta\left(t\right)}(\hat{\bf Z})(\phi_\delta(x)) = \frac{\cof \nabla\phi_\delta(x)}{\det \nabla\Phi_\delta(x)}\hat{\bf Z}(x)
 \]
 and set
\[\mathbf{Z}_{k}(y):=\mathcal{J}_{\delta\left(t\right)}\hat{\mathbf{Z}}_{k}(\phi_\delta(x))\text{ for }y=\phi_\delta(x).
\]
We remark here, that $\det \nabla\Phi_\delta>0$ if $\norm{\delta}_\infty\leq L$. If moreover $\nabla\delta\in L^\infty(\omega)$, the Piola transform is continuous between all $L^p$-spaces. If $\delta\in H^2$ and $\norm{\delta}_\infty\leq L$, then there is a loss in the continuity. Indeed, it can be readily checked that
\begin{align}
    \label{eq:piolacont}
    \norm{\mathcal{J}_{\delta\left(t\right)}\hat{f}}_{L^p_x}\leq c\norm{\hat{f}}_{L^q_x}\text{ and }\norm{\mathcal{J}_{\delta\left(t\right)}^{-1}f}_{L^p_x}\leq c\norm{f}_{L^q_x}
\end{align},
for all $q>p\geq 1$.
 %As the Piola transform can be seen as a linear bijective operator we find that for any $t\in [0,T]$, the set $\{\mathcal{J}_{\delta\left(t\right)}\hat{\mathbf{Z}}_{k}\}_{k\in \mathbb{N}}$ is a basis of $\left\{ \varphi\in H^{1}\left(\Omega_{\delta(t)}\right)\,:\,\divergence(\mathbf{\varphi})=0,\,\text{tr}_{\Gamma_p}\left(\mathbf{\varphi}\right)\cdot \tau_1=0=\text{tr}_{\Gamma_p}\left(\mathbf{\varphi}\right)\cdot \tau_2=\text{tr}_{\partial\Omega_{\delta(t)}\setminus\Gamma_P}(\mathbf{\varphi})\right\}$ (compare with \cite{LR14}).
 Finally, we join the families $\mathbf{Y}_{k}$ and $\mathbf{Z}_k$  into the formula
 \begin{equation}\label{eqn:Xk-def}
 \mathbf{X}_{k}=\begin{cases}
\mathbf{Y}_{k} & k\ \text{odd}\\
\mathbf{Z}_{k} & k\ \text{even}.
\end{cases}
 \end{equation}
Please observe that the corresponding traces
 \begin{equation}
X_{k}\nu_{\delta}:=\text{tr}_{\delta}\left(\mathbf{X}_{k}\right),\quad Y_{k}\nu_{\delta}:=\text{tr}_{\delta}\left(\mathbf{Y}_{k}\right)
 \end{equation}
 which do have the correct form to fit to $\xi_k$ on the moving part of the boundary and are free in normal direction on the steady inflow/outflow part of $\Gamma_p$.
It can be readily checked that the set  of vectors $\left\{ {\bf X}_{k}\right\} _{1\le k\le n}$  is linear  independent and are a subset of $C^1(I,H^1(\Omega_\delta))\times C^\infty(\omega)$.

Let us check that 
\begin{equation}\label{eqn:claim-density}
\text{\text{span}\ensuremath{\left\{ \left(\varphi{\bf X}_{k},\varphi X_{k}\right):k\ge1,\varphi\in C^{1}_{\text{per}}\left(I\right)\right\} \ \text{is dense in}}}\ V_{T}^{\delta}.
\end{equation}
Let  $\left({\bf q},\xi\right)\in V_{T}^{\delta}$. We consider ${\bf q}_0={\bf q}-\mathcal{F}_{\delta\left(t\right)}(\xi)$ and approximate this function which has a zero trace on the time-dependent part, by elements of ${\bf Z}_k$. This can actually best be performed by approximating the inverse of the Piola transform of $\mathcal{J}_{\delta\left(t\right)}{\bf q}_0$ by the elements of $\hat{\mathbf{Z}}_{k}$. The continuity and invertability of the Piola transform then produces a converging approximation. The reminder part $\mathcal{F}_{\delta\left(t\right)}(\xi)$ is approximated via $\xi$. Indeed, for $\varepsilon>0$, 
 there is some $n_0 \ge 1$ and some  functions $e_{k}(t)$,  $1\le k \le n$  for which $\xi_n(t,x):=\sum_{k=1}^{n}e_{i_{k}}(t)Y_{k}(x)$ satisfies
$\left\Vert \xi-\xi_n\right\Vert _{W^{1,2}(0,T;H^{2}\left(\omega\right)}\le\varepsilon$ for all $n\geq n_0$. 
%Consider now the function $\mathbf{q}_{n}\left(t,x\right):=\sum_{k=1}^{n}e_{i_{k}}\left(t\right)\mathbf{Y}_{k}\left(t,x\right)$.  Since $\mathbf{q-\mathbf{q}}_{n}-\mathcal{F}_{\delta\left(t\right)}\left(\xi-\xi_{n}\right)$ has trace zero on the moving part of $\partial \Omega_{\delta}$ there exist some scalars $f_k $ for which 
%\[
%\mathbf{q-\mathbf{q}}_{n}-\mathcal{F}_{\delta\left(t\right)}\left(\xi-\xi_{n}\right)=\sum_{k=1}^{\infty}f_{k}\mathbf{Z}_{k}\]
As $\mathcal{F}_{\delta\left(t\right)}$ is linear, we find 
\[
\left\Vert \mathcal{F}_{\delta}(\xi)-\sum_{k=1}^{n}e_{i_{k}}{\bf Y}_{k}\right\Vert_{W^{1,2}(0,T;W^{1,2}(\Omega_\delta))}=\left\Vert \mathcal{F}_{\delta\left(t\right)}\left(\xi-\xi_{n}\right)\right\Vert _{W^{1,2}(0,T;W^{1,2}(\Omega_\delta))}\lesssim\left\Vert \xi-\xi_{n}\right\Vert _{W^{1,2}(0,T;H^2\left(\omega\right))}\le\varepsilon,
 \]
which shows the density. 
\subsubsection{The Galerkin approximation} 
We  make the Galerkin ansatz
\begin{equation}\label{eqn:ansatz-u-n-eta-n}
\begin{aligned}\delta_{n}\left(t,x\right):= & \sum_{k=1}^{n}{\bf b}_{n}^{k}\left(t\right)X_{k}\left(x\right),\quad\left(t,x\right)\in I\times\omega\\
{\bf v}_{n}\left(t,\cdot\right):= & \sum_{k=1}^{n}\left({\bf b}_{n}^{k}\right)^{\prime}\left(t\right){\bf X}_{k}\left(t,\cdot\right),\quad\left(t,x\right)\in\Omega_{\delta}^{I}\\
\eta_{n}\left(t,x\right):= & \sum_{k=1}^{n}{\bf a}_{n}^{k}\left(t\right)X_{k}\left(x\right),\quad\left(t,x\right)\in I\times\omega\\
{\bf u}_{n}\left(t,\cdot\right):= & \sum_{k=1}^{n}\left({\bf a}_{n}^{k}\right)^{\prime}\left(t\right){\bf X}_{k}\left(t,\cdot\right),\quad\left(t,x\right)\in\Omega_{\delta}^{I},
\end{aligned}
\end{equation}
where $\delta_n$ and $\mathbf{v}_n$  are given and  $\eta_n$ and $\mathbf{u}_n$ are to be constructed.
We provide an initial data
 \begin{equation}\label{eqn:cond-initial-value-small}
{\bf b}_{n}\left(0\right)=:{\bf b}_{n0},{\bf b}_{n}^{\prime}\left(0\right)=:{\bf b}_{n1}
 \end{equation}

Please observe that \eqref{eqn:ansatz-u-n-eta-n} also incorporates the  condition
\begin{equation}\label{eqn:coupling-u-n-eta-n}
\text{tr}_{\delta}\left(\mathbf{u}_{n}\right)=\partial_{t}\eta_{n}\nu
\end{equation}
 as the functions $X_k$  depend only on $x$, by construction. 

 The $n^{\text{th}}$-dimensional approximate problem reads as follows:
 
 \begin{equation}\label{eqn:galerkin}
\begin{aligned}
&\frac{d}{dt}\int_{\Omega_{\delta\left(t\right)}}\frac{{\bf u}_{n}\cdot{\bf X}_{k}}{2}dx+\int_{\Omega_{\delta\left(t\right)}}\frac{\partial_{t}{\bf u}_{n}\cdot{\bf X}_{k}-{\bf u}_{n}\cdot\partial_{t}{\bf X}_{k}+\left({\bf v}_{n}\cdot\nabla\right){\bf u}_{n}\cdot{\bf X}_{k}-\left({\bf v}_{n}\cdot\nabla\right){\bf X}_{k}\cdot{\bf u}_{n}}{2} 
\\
&\quad +\int_{\Omega_{\delta\left(t\right)}}\nabla{\bf u}_{n}:\nabla{\bf X}_{k}dx+\int_{\Gamma_{p}}P{\bf X}_{k}\cdot\nu dA+\int_{\omega}\partial_{tt}\eta_{n}X_{k}+\varepsilon\nabla^3\eta_{n}\nabla^3 X_k dA+K\left(\eta_{\eta},X_{k}\right)
\\
&=\int_{\Omega_{\delta\left(t\right)}}{\bf f}\cdot{\bf {\bf X}}_{k}dx+\int_{\omega}gX_{k}dA\quad k=\overline{1,n},\ t\in I & .
\end{aligned}
\end{equation}
Now \eqref{eqn:galerkin} reads as a system of second order integro-differential equations with the unknown\footnote{see also \cite[p. 239]{LR14}} \[{\bf a}\left(t\right):=\left(\mathbf{a}_{n}^{k}\left(t\right)\right)_{k=1}^{n}:I\mapsto\mathbb{R}^{n}.\]
Note that the mass matrix is given by
\[M\left(t\right):=\left(\int_{\Omega_{\delta\left(t\right)}}\mathbf{X}_{i}\left(t\right)\cdot\mathbf{X}_{j}\left(t\right)dx+\int_{\omega}X_{i}X_{j}dA\right)_{1\le i,j\le2},
\]
which is positive definite for each $t\in I$. 
Now,  usual   Picard-Lindel\"{o}f  arguments can be applied to ensure the  local existence of $\mathbf{a}_n\left(t\right),\ t\in\left[0,t_{0}\right]$,
The extension of $\mathbf{a}_n$ to the whole $[0,T]$  is justified by the balance of energy, which is obtained directly by testing with $(\bu_n,\partial_t\eta_n)$ and which reads
\begin{equation}\label{eqn:energy-balance-E-n}
\frac{d}{dt}E_{n}\left(t\right)+\int_{\Omega_{\delta\left(t\right)}}\left|\nabla{\bf u}_{n}\right|^{2}dx+\varepsilon\int_{\omega}\abs{\nabla^3\partial_t\eta_n}^2dA+\int_{\Gamma_{p}}P\mathbf{u}_{n}\cdot\nu dA=\int_{\Omega_{\delta\left(t\right)}}{\bf f}\cdot{\bf u}_{n}dx+\int_{\omega}g\partial_{t}\eta_{n}dA,
\end{equation}
where we have denoted the energy $E_n$ by
\begin{equation}\label{eqn:energy-E-n}
 E_{n}\left(t\right):=\frac{1}{2}\int_{\Omega_{\delta_{n}\left(t\right)}}\left|{\bf u}_{n}\left(t,\cdot\right)\right|^{2}dx+\frac{1}{2}\int_{\omega}\left|\partial_{t}\eta_{n}\left(t,\cdot\right)\right|^{2}dA+K\left(\eta_{n}\left(t,\cdot\right)\right)\quad t\in I
\end{equation}
and thus by the arguments in the proof of Lemma~\ref{lem:energy} we find
\[
\sup_{t\in I}E_{n}\left(t\right)\le E_{n}\left(0\right)+ C,
\]
for some $C$ depending on the data and the existence over the interval $I=[0,T]$ follows.

\subsubsection*{Establishing the periodicity}
Consider the mapping 
\begin{equation}
\mathcal{P}:\mathbb{R}^{2\times n}\mapsto\mathbb{R}^{2\times n},\ \mathcal{P}:\left(\mathbf{a}_{n0},\mathbf{a}_{n1}\right)\mapsto\left(\mathbf{a}_{n}\left(T\right),\mathbf{a}_{n}^{\prime}\left(T\right)\right)
\end{equation}
and note that it is well-defined in view of the above-discussion. We aim to prove that $\mathcal{P}$ has at least one fixed-point and we employ the Leray-Schauder theorem.
Please note that $\mathcal{P}$ is continuous, since the differential system \eqref{eqn:galerkin} enjoys the property that the solutions  are continuous w.r.t. the initial data. It is easy to check that $\mathcal{P}$ is also compact.  Now, we consider the set 
\begin{equation}
    LS:=\left\{ \left(\mathbf{a}_{n0},\mathbf{a}_{n1}\right):\left(\mathbf{a}_{n0},\mathbf{a}_{n1}\right)=\lambda\mathcal{P}\left(\mathbf{a}_{n0},\mathbf{a}_{n1}\right)\ \text{for some }\lambda\in\left[0,1\right]\right\} 
\end{equation}
and we claim that it is uniformly bounded w.r.t. $\lambda$. To this end, let $\left(\mathbf{a}_{n0},\mathbf{a}_{n1}\right)\in LS$. The corresponding solution $\left(\mathbf{u}_{n},\eta_{n}\right)$ enjoys the property 
\begin{equation}
    \left(\mathbf{u}_{n}\left(0\right),\eta_{n}\left(0\right),\partial_{t}\eta_{n}\left(0\right)\right)=\lambda\left(\left(\mathbf{u}_{n}\left(T\right),\eta_{n}\left(T\right),\partial_{t}\eta_{n}\left(T\right)\right)\right)
\end{equation}
for some $\lambda\in\left[0,1\right]$. We exclude the trivial case $\lambda=0$.
This yields 
\begin{equation}
    E_{n}\left(T\right)=\frac{1}{\lambda^{2}}E_{n}\left(0\right)\ge E_{n}\left(0\right).
\end{equation}
Hence \eqref{eqn:energy-balance-E-n} implies that by the argument of Corollary~\ref{cor:gal} we find
\begin{equation}\label{eqn:sup-E-n-t}
    \sup_{t\in I}E_{n}\left(t\right)\le C
\end{equation}
and in particular $E_{n}\left(0\right)\le C$ and the set $LS$ is uniformly bounded. Indeed, we can repeat the proof, as both the energy inequality, as well as the use of the testfunction $(\Fcal_\delta \eta_n,\eta_n)$ is valid on the Galerkin level.
%This argument follows the same lines as in \cite[Proposition 3.22]{our-paper}.
Thus the mapping $\mathcal{P}$ has at least one fixed point, that is there exists an initial data $\left(\mathbf{a}_{n0},\mathbf{a}_{n1}\right)\in\mathbb{R}^{2\times n}$ for which 
\begin{equation}
    \left(\mathbf{u}_{n}\left(0\right),\eta_{n}\left(0\right),\partial_{t}\eta_{n}\left(0\right)\right)=\left(\mathbf{u}_{n}\left(T\right),\eta_{n}\left(T\right),\partial_{t}\eta_{n}\left(T\right)\right).
\end{equation}

\subsection{Recovering the coupling}
\label{sec:fp2}

Consider  $\mathcal{B}_{n}$ a convex subset of $C_{\text{per}}^{1}\left(I,\mathbb{R}^n\right)$ and the operator 
\begin{equation}
\mathcal{S}:\mathcal{B}_n\mapsto 2^{\mathcal{B}_n},\ \mathcal{S}:\mathbf{b}_{n}\mapsto\left\{ \mathbf{a}_{n}:\ \mathbf{a}_{n}\ \text{is a time-periodic solution to \eqref{eqn:galerkin}}\right\} 
\end{equation}
We aim to apply the Glicksberg-Kakutani-Fan fixed point theorem to prove that there exists at least one $\mathbf{b}_{n}\in C_{\text{per}}^{1}\left(I\right)$ for which $\mathbf{b}_{n}\in\mathcal{S}\left(\mathbf{b}_{n}\right)$. This corresponds to a recovering of the coupling condition.

First we specify the set $\mathcal{B}_n$. For that we fix $\tilde{M}>0$, such that 
\begin{align}
\label{eq:condition1}
E(t)\leq \tilde{M}^2\text{ implies }\norm{\eta(t)}_\infty \leq L, 
\end{align}
as mentioned before the size of the Lam\'e constants have a significant impact in the size of $\tilde{M}$. The size of $\tilde{M}$, then implies the size of $\tilde{C}$ by \eqref{cond2}. (Here $\theta\in (0,1)$ can be optimized in accordance with the other fixed constants).

Next, we associate with $\mathbf{b}_n$, the function $\delta_n:=\sum_{k=1}^n b_n^k X_k$ and define
\[
\mathcal{B}_n:=\big\{ \mathbf{b}_n\in C_{\text{per}}^{1}\left(I,\mathbb{R}^n\right)\, :\, \norm{\delta_n}_{L^\infty_{t,x}}\leq L\text{ and } \alpha\left\Vert \delta\right\Vert _{L_{t}^{\infty}W_{x}^{1,2}}^2+\frac{1}{2}\left\Vert \partial_{t}\delta\right\Vert _{L_{t}^{\infty}L_{x}^{2}}^2\leq \tilde{M}^2\big\},
\]
where $\alpha$ is chosen according to \eqref{eq:alpha}. 

\begin{itemize}
    \item The set $\mathcal{B}_n$ is convex.
    \item $\mathcal{S}$ is upper-semicontinuous, which is equivalent to the closed-graph property; this is true as the system \eqref{eqn:galerkin} is linearized.
    \item We have the compact embedding $\mathcal{S}\left(\mathcal{B}_n\right)\subset C_{\text{per}}^{2}\left(I\right)$ which is due to Arzela-Ascoli theorem. Indeed, consider, for fixed $n$, a bounded sequence $\sup_{k}\left\Vert \mathbf{b}_{n_{k}}\right\Vert_{C_{t}^{1}} \le1$. 
    Then, for each $\mathbf{a}_{n_{k}}\in\mathcal{S}\left(\mathbf{b}_{n_{k}}\right)$ the estimate \eqref{eqn:R} implies that $\sup_{k \ge 1}\left\Vert \mathbf{a}_{n_{k}}\right\Vert _{C_{t}^{1}}\le c$. Now,  please observe that if we take $C^2_{t}$ norms in  \eqref{eqn:galerkin} we obtain that  $\sup_{k}\left\Vert \mathbf{a}_{n_{k}}\right\Vert _{C_{t}^{2}}\le C$ where $C$ depends on $n$,  $\max_{1\le i\le n}\left\Vert \left(\mathbf{X}_{i},X_{i}\right)\right\Vert _{C_{t,x}^{1}}$, $\delta_n$ and so on, but not on $k$. The Arzela-Ascoli theorem ensures that the sequence $\left(\mathbf{a}_{n_{k}}\right)_{k}$ has a convergent subsequence in $C^{1} _ {\text{per}}(I)$.

    \item For any $\mathbf{b}_n\in \mathcal{B}_n$ we have that $\mathcal{S}\left(\mathbf{b}_{n}\right)$ is nonempty,  convex (the problem is linearized) and compact (Arzela-Ascoli).
    \item For any $\mathbf{b}_n\in \mathcal{B}_n$ we have that $\mathcal{S}\left(\mathbf{b}_{n}\right)\subset\mathcal{B}_n$, which follows by the arguments in Corollary~\ref{cor} and the above choices of $\tilde{M}$ (in dependence of $L$, $\kappa$ and $\omega$) and $\tilde{C}$ (in dependence of $\tilde{M},L$ and the curvatures and the length of the reference geometry $M$).%, such that $\mathbf{b}_n\in \mathcal{B}_n$.
\end{itemize}
Thus, by applying the Kakutani-Glicksberg-Fan fixed point theorem we ensure that there exists at least one $\mathbf{a}_{n}\in C_{\text{per}}^{2}\left(I,\mathbb{R}^n\right)\cap \mathcal{B}_n$ for which $\mathbf{a}_{n}\in\mathcal{S}\left(\mathbf{a}_{n}\right)$. This fixed point is the desired coupled solution.

\subsubsection{Letting $n \to \infty$ }
By the energy estimate  it follows that there exists a pair $\left(\mathbf{u},\eta\right)\in V_{S}^{\eta}$ and a  convergent subsequence (not relabeled) for which
\begin{equation}
\begin{aligned}\mathbf{u}_{n}\to\mathbf{u} & \ \text{weakly* in}\ L_{t}^{2}L_{x}^{2}\cap L_{t}^{\infty}L_{x}^{2}\\
\partial_{t}\eta_{n}\to\partial_{t}\eta & \ \text{weakly* in}\ L_{t}^{2}H_{x}^{2}\cap L_{t}^{\infty}L_{x}^{2}
\end{aligned}
\end{equation}
Further, by Arzel\'{a}-Ascoli, we find that 
\[
\eta_n\to \eta\text{ strongly in } C^\alpha([0,T]\times \omega), 
\]
for some $\alpha>0$.
The proof of the \emph{strong convergence} of this subsequence, that is 
 \begin{equation}\label{eqn:strong-conv-u-n}
     \left(\mathbf{u}_{n},\partial_{t}\eta_{n}\right)\to\left(\mathbf{u},\partial_{t}\eta\right)\ \text{in }L_{t}^{2}L_{x}^{2}\times L_{t}^{2}L_{x}^{2}.
 \end{equation}
 is presented in Subsection~\ref{subsection:L2-compact}.
Note that from \eqref{eqn:strong-conv-u-n} and  \eqref{eqn:coupling-u-n-eta-n} we get  
\begin{equation}
\text{tr}_{\eta}\left(\mathbf{u}\right)=\partial_{t}\eta\nu.
\end{equation}
Let us now multiply \eqref{eqn:galerkin} by $\varphi\in C_{\text{per}}^{1}\left(I\right)$ and integrate by parts  to obtain that
 \begin{equation}
 \label{eq:wfd}
\begin{aligned}\int_{0}^{T}\int_{\Omega_{\eta_{n}\left(t\right)}}-{\bf u}_{n}\cdot\partial_{t}\mathbf{q}_{n}+\frac{1}{2}\left(\mathbf{u}_{n}\cdot\nabla\right)\mathbf{u}_{n}\cdot\mathbf{q}_{n}-\frac{1}{2}\left({\bf u}_{n}\cdot\nabla\right)\mathbf{q}_{n}\cdot{\bf u}_{n}+\nabla{\bf u}_{n}:\nabla\mathbf{q}_{n}dxdt & +\\
\int_{0}^{T}\int_{\Gamma_{p}}P{\bf q}_{n}\cdot\nu dA+\int_{0}^{T}\int_{\omega}\varepsilon\nabla^{3}\partial_{t}\eta_{n}\nabla^{3}\xi_{n}-\partial_{t}\eta_{\eta}\partial_{t}\xi_{n}dAdt+2\int_{0}^{T}K\left(\eta_{\eta},\xi_{n}\right)dt & =\\
\int_{0}^{T}\int_{\Omega_{\eta\left(t\right)}}{\bf f}\cdot\mathbf{q}_{n}dxdt+\int_{\omega}g\xi_{n}dAdt
\end{aligned}
\end{equation}
 for all $\left(\mathbf{q}_{n},\xi_{n}\right)\in\text{span}\left\{ \left(\varphi\mathbf{X}_{k},\varphi X_{k}\right):\varphi\in C_{\text{per}}^{1}\left(I\right)1\le k\le n\right\} $.
By letting $n\to \infty$, taking into account \eqref{eqn:claim-density}, we obtain exactly the relation~\eqref{eqn:epsilonlevel}, which concludes Theorem \ref{thm:epsilon}.

\subsection{$L^2$-compactness}\label{subsection:L2-compact}
\label{sec:L2comp}
% At this level, the stabilizer $\varepsilon \Delta^2\partial_t\eta_n$ helps, as it implies that $(\partial_t\eta_n(t),\bu_n(t))$ is a suitable test-function for \eqref{eq:wfd}. Hence we can use Theorem \ref{thm:auba}, which is a particular simple form of Theorem~\cite[Theorem 5.1]{MS22}. 

% We define $g_n=(\partial_t \eta^{n},\bu^{n}\chi_{\Omega_{\eta^{n}}})$ and $f_n=(\partial_t \eta^{n},\bu_n)$, where we may extend $\bu$ naturaly to $\Omega_L$ by $\testn(\partial_t\eta_n)$. For these we apply Theorem~\ref{thm:auba} using the spaces $X={L^2(\omega)\times L^2(\Omega_L)}$ and consequently $X'={L^2(\omega)\times L^2(\Omega_L)}$.  %The compactness space is $Y={H^{s}(\omega)\times L^2(Q^\kappa)}$,
% The space $Z=H^{2}(\omega)\times H^{1}(\Omega_L)$. Also with respect to time we work in the setting of Hilbert spaces, which means that all Lebesgue exponents are equal to two. 

%We follow the approach of \cite[Subsection 6.2]{MS22}, which means applying Theorem~\cite[Theorem 5.1]{MS22} two times: Once for the part of the velocity that is related to the impact of the moving boundary, once for the other parts. Here the deviation of the basis according to the impact becomes handy. 

We follow that approach as in the proof of \cite[Lemma 6.2]{MS22}. For that we split the sequence in two parts for which we will show independently compactness:
\begin{align*}
\int_0^T\norm{\partial_t \eta^{n}}_{L^2(\omega)}^2+\norm{
\bu^{n}}_{L^2(\Omega_{\eta^{n}})}^2\,dt
&=\int_0^T\skp{\partial_t \eta^{n},\partial_t \eta^{n}}+\skp{\bu^{n},\testn(\partial_t \eta^{n})}\, dt
\\
&\quad +\int_0^T\skp{\bu^{n},\bu^{n}-\testn(\partial_t \eta^{n})}\, dt=:(I_n)+(II_n).
\end{align*}
For both terms we will show the convergence separately by using Theorem \ref{thm:auba}.
Their convergence then implies that 
\[
\norm{\partial_t \eta^{n}}_{L^2_tL^2_x}+\norm{\bu^{n}\chi_{\Omega_{\eta^{n}}}}_{L^2_tL^2_x}\to \norm{\partial_t \eta}_{L^2_tL^2_x}+\norm{\bu\chi_{\Omega_{\eta}}}_{L^2_tL^2_x},
\]
which implies, by the uniform convexity of $L^2$, the strong convergence $(\partial_t\eta^{\varepsilon},\bu^{\varepsilon})\to (\eta,\bu)$ and the Lemma is proved.

{\bf For $(I_n)$ we define} $g_n=(\partial_t \eta^{n},\bu^{n}\chi_{\Omega_{\eta^{n}}})$ and $f_n=(\partial_t \eta^{n},\testn(\partial_t \eta^{n})))$, that we may extend naturally to $\Omega_L$ by the extension. For these we apply Theorem~\ref{thm:auba} using the spaces $X={L^2(\omega)\times L^2(\Omega_L)}$ and consequently $X'={L^2(\omega)\times L^2(\Omega_L)}$.  %The compactness space is $Y={H^{s}(\omega)\times L^2(Q^\kappa)}$,
The space $Z=H^{s_0}(\omega)\times H^{s_0}(S_L)$ for $0<s<s_0<\frac14$. Also with respect to time we work in the setting of Hilbert spaces, which means that all Lebesgue exponents are equal to two. 

Particular is here the choice of the approximation $f_{n,\epsilon}$. For that we extend all ${\bf Y}_k$ to $\Omega_L$, as $\Fcal_{\eta_n}({Y}_k)$ is defined there. Now, by the properties of the basis $Y_k$, we can for each $\epsilon>0$ truncate $\partial_t\eta_n=\sum_{k=1}^n{\bf a}_n^k X_k$ at the level of $m$ and define
\begin{align*}
f_{n,m}:=(\partial_t\eta_{n,m}, \testn \partial_t\eta_{n,m}):=\Big(\sum_{k\leq m} {\bf a}_n^k \xi_k, \sum_{k\leq m} {\bf a}_n^k \testn\xi_k\Big).
\end{align*}
By \eqref{eq:unif} and Proposition \ref{prop:Solenoidal-extension-o} we find an $n_\epsilon$, such that for all $m\geq n_\epsilon$
\begin{align}
\label{eq:2}    
\norm{f-f_{n,m}}_{L^2_t(L^2_x)}\leq \epsilon,\text{ and }
\norm{f_{n,m}}_{W^{1,2}_tW^{1,2}_x}\leq C(m),
\end{align}
as $\eta_n$ is uniformly bounded in $L^\infty_tH^2_x$.

%\[
%Z^n(t)=\{(\xi, \bphi)\in H^3(\omega)\times H^2(Q^\kappa)\, :\,\diverg \bphi=0,\quad \bphi\circ \bvphi_{\eta^{n}(t)}=\xi{\bn}\}.
%\] 
%Please observe that $f_n$ is uniformly bounded in $Z$, which follows by the trace estimate. Since $(\partial_t \eta^{n},\testn(\partial_t \eta^{n}))$ are valid test-functions on the discrete level, 

Next, let us check the assumptions of  Theorem~\ref{thm:auba}. First observe that  \cite[Proposition 2.28]{LR14} 
implies that $g_n$ is uniformly bounded in $L^2_t(H^{s}_x)$ (for $s\leq \frac14$).
Hence the assumptions (1) follows in a rather straightforward manner by weak compactness in Hilbert spaces. 
Next 
(2) follows from \eqref{eq:2}.
%from Proposition~\ref{TestFunction} and the a-priori estimate on $\partial_t \eta\in L^2_tH^{s}_x$ for $s<\frac12$.
%For (3') we will use $(\partial_t \eta^{n})_\delta= \partial_t \eta^{n}*\psi_\delta$ , where $\psi_\delta$ is the standard convolution kernel in space and ${f_{n,\delta,t}(\sigma)=(
%(\partial_t \eta^{n}(\sigma)*\psi_\delta,E_{\eta^{n}(t),\delta}(\partial_t \eta^{n}(\sigma)))}$, where $E_{\eta^{n}(t),\delta}$ is the smooth extension introduced in Corollary~\ref{cor:extend}. These definitions imply the necessary bounds. In particular, we obtain the $L^2$ convergence as a consequence of 
As (4) is covered by our choices of spaces, we are left to check (3). As usual for equi-continuity in time, this is a consequence of the weak formulation of the problem \eqref{eq:wfd}. For $ \sigma\in [t,t+\tau]$ (using the solenoidality and the matching of the extension) we have 
\begin{align*}
\abs{\skp{g_n(t)-g_n(\sigma),f_{n,m}(t)}}
&=\absBB{\int_{\Omega_L} \Big(\bu_{n}(t)\chi_{\Omega_{\eta_{n}(t)}}-\bu_{n}(\sigma)\chi_{\Omega_{\eta_{n}(\sigma)}}\Big)\cdot \mathcal{F}_{\eta_n(t)}(\partial_t\eta_{n,m}(t))\, dx
\\
&\quad   +\int_{\omega} \big(\partial_t\eta^{n}(t)-\partial_t\eta^{n}(\sigma)\big)\partial_t\eta_{n,m}\, d\by}
\\
&=\absBB{\int_{\Omega_L} \bu_{n}(t)\chi_{\Omega_{\eta_{n}(t)}}\cdot \mathcal{F}_{\eta_n(t)}(\partial_t\eta_{n,m}(t))-\bu_{n}(\sigma)\chi_{\Omega_{\eta_{n}(\sigma)}}\cdot \mathcal{F}_{\eta_n(\sigma)}(\partial_t\eta_{n,m}(t))dx 
\\
&\quad  +\int_{\omega} \big(\partial_t\eta_{n}(t)-\partial_t\eta_{n}(\sigma)\big)\partial_t \eta_{n,m}(t))\, d\by}
\\
&\quad +\absBB{\int_{\Omega_L} \bu_{n}(\sigma)\chi_{\Omega_{\eta_{n}(\sigma)}}\cdot \Big(\mathcal{F}_{\eta_n(t)}(\partial_t\eta_{n,m}(t))-\mathcal{F}_{\eta_n(\sigma)}(\partial_t\eta_{n,m}(t))\Big)dx }
\\
&=: (A_1)+(A_2)
\end{align*}
First observe that by Proposition \ref{prop:Solenoidal-extension-o}
\begin{align*}
(A_2)&\leq \int \abs{\bu_{n}(\sigma)}\chi_{\Omega_{\eta_{n}(\sigma)}}\cdot \int_t^\sigma\abs{\partial_s \mathcal{F}_{\eta_n(s)}(\partial_t\eta_{n,m}(t))}\, ds \, dx 
\\
&\leq c\norm{\bu_{n}(\sigma)}_{L^2(\Omega_{\eta_{n}(\sigma)})}\norm{\partial_t \eta_{n}(\partial_t \eta_{n,m})(t)}_{L^\infty_t(L^2_x)}\abs{\sigma-t}^\frac12
\\
&\leq C(m)\tau^\frac12,
\end{align*}
where we used the uniform $L^\infty_tL^2_x$ estimate of $\bu_{n}$ and $\partial_t \eta_{n}$.

Second, by the discrete-in-space weak formulation \eqref{eq:wfd}, we find
\begin{align*}
(A_1)
&=\absBB{\int_\sigma^t\int_{\Omega_\eta}-\bu_{n}(s)\cdot \partial_s \Fcal_{\eta_{n}(s)}(\partial_t \eta_{n,m}(t))
\\
&+(\nablasym\bu_{n}(s)-\bu_{n}(s) \otimes \bu_{n} (s)) :\nabla  \Fcal_{\eta_{n}(s)}(\partial_t \eta_{n,m}(t))\, dx 
\\
& 
\quad +\int_{\omega}2K(\eta_{n},(\partial_t \eta_{n,m}(t)))dA\, ds}
\\
&\leq C(m)\Big((\norm{\bu_{n}}_{L^2(t,t+\tau;W^{1,2})(\Omega_{\eta_{n}})}(\norm{\partial_t \eta_{n}}_{L^\infty_t(L^2_x)}+\norm{\eta_{n}}_{L^\infty_t(H^2_x)})
\\
&\quad +\norm{\eta_{n}}_{L^2(t,t+\tau;W^{1,\infty}(\omega)}\norm{\bu_{n}}_{L^\infty(0,T;L^{2}(\Omega_{\eta_{n}}))}^2\Big) \abs{t-\sigma}^\frac12 
\\
&\leq C(m)\tau^\frac12.
\end{align*}
This implies (3), namely
\[
\Big|\fint_t^{t+\tau}\skp{g_n(t)-g_n(\sigma),f_{n,m}(t)}\, d\sigma\Big|\leq C(m)\tau^\frac12.
\]
This finishes the proof of the convergence of $(I)_n$ term. 

{\bf For the second term $(II_n)$} we again apply Theorem~\ref{thm:auba}.
Here we set $g_n=\bu_{n}\chi_{\Omega_{\eta_{n}}}$ and $f_n=(\bu_{n}-\testn(\partial_t \eta_{n}))\chi_{\Omega_{\eta_{n}}}$ 
We apply Theorem~\ref{thm:auba} with the following spaces $X={H^{-s}(Q^\kappa)}$ and consequently $X'={H^{s}(\Omega_L)}$ for some $s\in (0,\frac14)$. Further we define $Z=L^2(\Omega_L)$. Please observe that we may extend all involved quantities by zero to be functions over $\Omega_L$. Finally, we again set all Lebesgue exponents to two. Similarly as for the first term, again the main effort is the construction of the right mollification. Indeed the assumptions $(1)$ and $(4)$ follow by standard compactness arguments. In particular, for assumption $(1)$ it has been shown in \cite[Proposition 2.28]{LR14} that $g_n$ is uniformly bounded in $H^{s}(Q^\kappa)$ (if $s\leq \frac14$).
For (2) we use the fact that $f_n$ has a zero trace on $\partial\Omega_{\eta_{n}(t)}$. First, let $\delta>0$ be given. We take $n_\delta$ large enough and $\tau_\delta>0$ small enough, such that 
\begin{align}
\label{eq:domain}
\sup_{n\geq n_\delta}\sup_{\tau\in (t-\tau_\delta,t+\tau_\delta)\cap [0,T]}\norm{\eta(t,x)-\eta_{n}(\tau,x)}_\infty\leq \delta.
\end{align}
Second, we fix $0<s<s_0<\frac{1}{4}$ and $\epsilon>0$. By \cite[Lemma A.13]{LR14} there exits a $\sigma_\epsilon$ and a sequence $\tilde{f}_{n,\delta}$, such that $\supp(\tilde{f}_{n,\delta}(t))\subset \Omega_{\eta^{n}(t)-3\delta}$ for all $3\delta\leq \sigma_\epsilon$,
 that is divergence free and 
$\norm{f-\tilde{f}_{n,\delta}}_{(H^{-{s_0}}(\Omega_L))}\leq \epsilon\norm{f}_{L^2(\Omega_L)}$.  And $\norm{\tilde{f}_{n,\delta}}_{L^2(\Omega_L))}\leq c\norm{f}_{L^2(\Omega_L)}$.
We mollify this solenoidal function to define
\[
f_{n,\delta}= \tilde{f}_{n,\delta}*\psi_\delta\text{ where $\psi$ is the standard mollifier in space.}
\]
Then this definition implies by a standard convolution estimate that
\begin{align*}
\norm{f_n-f_{n,\delta}}_{H^{-s}(\Omega_L)}&\leq \norm{\tilde{f}_{n,\delta}-f_{n,\delta}}_{H^{-s}(Q^\kappa)} + \norm{\tilde{f}_{n,\delta}-f_n}_{H^{-s_0}(Q^\kappa)}
\\
&\leq c\delta^{s-s_0}\norm{\tilde{f}_{n,\delta}}_{H^{-s_0}(Q^\kappa)}+\epsilon\norm{f_n}_{L^2(Q^\kappa)}\leq c\epsilon \norm{f_n}_{L^2(\Omega_L)},
\end{align*}
for $\delta$ small enough in dependence of $s-s_0$.
Please observe that by the properties of the mollification and \eqref{eq:domain} that $\supp(f_{n,\delta}(t))\subset \Omega_{\eta_{n}(t)-2\delta}\subset \Omega_{\eta_{m}-\delta}(t+\tau)$ for $\tau\in [-\tau_\delta,\tau_\delta]$ and $m\geq n_\delta$.  We define $\mathcal{J}_n$ as the Piola transform from $\Omega\to \Omega_{\eta_{n}}$. Next we fix $m_0\geq n_\delta$ and find that $\mathcal{J}_{m}^{-1}f_{n,\delta}$ is compactly supported in $\Omega$ for all $m\geq 0$. Hence we can project it to $\{\hat{\bf Z}_k\}_{k\leq m}$, where we choose $m_0$ large enough, such that there is an $\hat{f}_{n,m}=\sum \hat{f}_{n,m}^k \hat{\bf Z}_k$ with 
\[
\norm{\hat{f}_{n,m}-\mathcal{J}_n^{-1}f_{n,\delta}}_{L^2_x}\leq \frac{\epsilon}{C(\delta)} \norm{\mathcal{J}_{n}^{-1}f_{n,\delta}}_{H^1_x}\leq c\epsilon \norm{f_{n}}_{L^2(\Omega_L)},
\]
as
\[
\norm{\mathcal{J}_{m_0}^{-1}f_{n,\delta}}_{H^1_x}\leq C(\norm{\eta_{n}}_{H^2_x\cap W^{1,\infty}_x})\norm{f_{n,\delta}}_{W^{1,\infty}_x}\leq C(\delta)\norm{f_{n}}_{L^2(\Omega_L)},
\]
where we used the stabilisation term implying $\eta_n L^\infty_t(H^3_x)$
Now finally 
\[
f_{n,m}:=\mathcal{J}_n\hat{f}_{n,m}
\]
is a suitable approximation of $f_n$, relying again on the regularity of $\eta$. Indeed (2) follows, by the continuity of the Piola transform \eqref{eq:piolacont}, as by Sobolev embedding (on $\Omega_L$), there is some $p<2$, such that
\begin{align*}
    \norm{f_{n,m}-f_n}_{H^{-s}(\Omega_L)}&\leq \norm{f_{n,m}-f_{n,\delta}}_{H^{-s}(\Omega_L)}+\norm{f_{n}-f_{n,\delta}}_{H^{-s}}
    \leq c\norm{f_{n,m}-f_{n,\delta}}_{L^p}+\norm{f_{n}-f_{n,\delta}}_{H^{-s}}
    \\
    &\leq c\norm{\mathcal{J}_{n}^{-1}f_{n,m}-\mathcal{J}_n^{-1}f_{n,\delta}}_{L^2}+\norm{f_{n}-f_{n,\delta}}_{H^{-s}}\leq c\epsilon \norm{f_n}_{L^2}=c\epsilon,
\end{align*}
as $f_n=(\bu_{n}-\testn(\partial_t \eta_{n}))\in L^\infty_t(L^2_x)$.
Moreover, by the regularity of $\eta_n$ and by its definition $\mathcal{J}_{\eta_n(\sigma)}f_{n,\delta}(t)$ can be used as a testfunction on the fluid equation alone for $\sigma\in[t,t+\tau]$ and $t\in [0,T-\tau]$. 
Hence (3) follows exactly along the lines of the above estimates using the weak formulation \eqref{eq:wfd}. Notice that here we will work just with the fluid equation since the traces of test function are zero at the moving interface. Indeed,
\begin{align*}
\abs{\skp{g_n(t)-g_n(\sigma),f_{n,m}(t)}}
&=\absBB{\int_{\Omega_L} \Big(\bu_{n}(t)\chi_{\Omega_{\eta_{n}(t)}}-\bu_{n}(\sigma)\chi_{\Omega_{\eta_{n}(\sigma)}}\Big)\cdot \mathcal{J}_{\eta_n(t)}\hat{f}_{n,m}(t)\, dx}
\\
&=\absBB{\int_{\Omega_L} \bu_{n}(t)\chi_{\Omega_{\eta_{n}(t)}}\cdot \mathcal{J}_{\eta_n(t)}\hat{f}_{n,m}(t)
-\bu_{n}(\sigma)\chi_{\Omega_{\eta_{n}(\sigma)}}\cdot \mathcal{J}_{\eta_n(\sigma)}\hat{f}_{n,m}(t)dx }
\\
&\quad +\absBB{\int_{\Omega_L} \bu_{n}(\sigma)\chi_{\Omega_{\eta_{n}(\sigma)}}\cdot \Big(\mathcal{J}_{\eta_n(t)}\hat{f}_{n,m}(t)-\mathcal{J}_{\eta_n(\sigma)}\hat{f}_{n,m}(t)\Big)dx }
\\
&=: (B_1)+(B_2)
\end{align*}
The term $(B_1)$ is estimated analogously to $(A_1)$ 
The estimate of $(B_2)$ uses the a-priori regularity of $\partial_t\eta_n\in L^2(H^3)$:
\begin{align*}
(B_2)&\leq \int \abs{\bu_{n}(\sigma)}\chi_{\Omega_{\eta_{n}(\sigma)}}\cdot \int_t^\sigma\abs{\partial_s \mathcal{J}_{\eta_n(s)}(\hat{f}_{n,m}(t))}\, ds \, dx 
\\
&\leq c(L,\norm{\nabla \eta_n}_\infty)\norm{\bu_{n}(\sigma)}_{L^2_x}\norm{\partial_t \nabla \eta_{n}}_{L^2_t(L^2_x)}\abs{\sigma-t}^\frac12\norm{\hat{f}_{n,m}(t))}_{L^\infty_x}
\\
&\leq C(m)\tau^\frac12,
\end{align*}
as $\norm{\hat{f}_{n,m}(t))}_{L^\infty_x}\leq c(m) \norm{\bu-\testn(\partial_t\eta_n)}_{L^2_x}\leq c$. This finishes the proof of the $L^2$ compactness.

\subsection{Proof of Theorem \ref{thm:main} and Theorem~\ref{thm:cauchy}}
First Theorem \ref{thm:main} follows precisely by the lines of the $\varepsilon\to 0$ limit proofed in \cite[Section 6]{MS22}.

For Theorem~\ref{thm:cauchy} builds on \eqref{eqn:galerkin}. One then performs directly the fixed point of Subsection~\ref{sec:fp2}. Here the standard energy estimate suffices for a $\delta_n$-independent a-priori estimate. This induces a convex, closed, compact set on which a coupled Galerkin solution can be found. The limit passage with the coupled Galerkin solution and the limit $\epsilon\to0$ is then exactly the same as in the time-periodic case. The solution can then be iteratively extended in time, by reformulating the respective barriers up to a time instant where a topological change of the domain is appearing.

\section*{Acknowledgments}
C. M. and S. S. thank the support of the ERC-CZ Grant LL2105 CONTACT of the Faculty of Mathematics and Physics of Charles University.

Moreover, C.M.  acknowledges for the support of the project GAUK No.  456120, of Charles University.
S.S. thanks for the support of  the University Centre UNCE/SCI/023 of Charles University and the VR Grant 2022-03862 of the Swedish Science Foundation.

\bibliographystyle{acm}
\bibliography{references}

\begin{thebibliography}{10}

\bibitem{BaSt21}
{\sc Bathory, M., and Stefanelli, U.}
\newblock Variational resolution of outflow boundary conditions for
  incompressible {Navier}-{Stokes}.
\newblock {\em Nonlinearity 35}, 11 (2022), 5553--5592.

\bibitem{FSIforBIO}
{\sc Bodnar, T., Galdi, G.~P., and Necasova, S.}, Eds.
\newblock {\em Fluid-Structure Interaction and Biomedical Applications}.
\newblock Birkh{\"a}user/Springer, Basel, 2014.

\bibitem{Bog}
{\sc Bogovskij, M.~E.}
\newblock The solution of some problems of vector analysis related to the
  operators div and grad.
\newblock Tr. {Semin}. {S}.{L}. {Soboleva} 1, 5-40 (1980)., 1980.

\bibitem{bonheure2019periodic}
{\sc Bonheure, D., Gazzola, F., and Dos~Santos, E.~M.}
\newblock Periodic solutions and torsional instability in a nonlinear nonlocal
  plate equation.
\newblock {\em SIAM Journal on Mathematical Analysis 51}, 4 (2019), 3052--3091.

\bibitem{BreSch18}
{\sc Breit, D., and Schwarzacher, S.}
\newblock Compressible fluids interacting with a linear-elastic shell.
\newblock {\em Arch. Rat. Mech. Anal. 228\/} (2018), 495--562.

\bibitem{breit-schwarz-fourier}
{\sc Breit, D., and Schwarzacher, S.}
\newblock Navier-stokes-fourier fluids interacting with elastic shells.
\newblock {\em accepted at Annali della Scuola Normale di Pisa - Classe di
  Scienze, arXiv:2101.00824\/} (2020).

\bibitem{C21}
{\sc {\v{C}}ani{\'c}, S.}
\newblock Moving boundary problems.
\newblock {\em Bull. Am. Math. Soc., New Ser. 58}, 1 (2021), 79--106.

\bibitem{Gr05}
{\sc Chambolle, A., Desjardins, B., Esteban, M.~J., and Grandmont, C.}
\newblock Existence of weak solutions for the unsteady interaction of a viscous
  fluid with an elastic plate.
\newblock {\em J. Math. Fluid Mech. 7}, 3 (2005), 368--404.

\bibitem{Ci05}
{\sc Ciarlet, P.~G.}
\newblock {\em An introduction to differential geometry with applications to
  elasticity}, reprinted from the {Journal} of {Elasticity} 78--79, {No}. 1-3
  (2005)~ed.
\newblock Dordrecht: Springer, 2005.

\bibitem{evans}
{\sc Evans, L.~C.}
\newblock {\em Partial differential equations}, vol.~19.
\newblock American Mathematical Society, 2010.

\bibitem{galdi2006existence}
{\sc Galdi, G., and Silvestre, A.}
\newblock Existence of time-periodic solutions to the navier--stokes equations
  around a moving body.
\newblock {\em Pacific journal of mathematics 223}, 2 (2006), 251--267.

\bibitem{galdi2016bifurcating}
{\sc Galdi, G.~P.}
\newblock On bifurcating time-periodic flow of a navier-stokes liquid past a
  cylinder.
\newblock {\em Archive for Rational Mechanics and Analysis 222\/} (2016),
  285--315.

\bibitem{galdi2020viscous}
{\sc Galdi, G.~P.}
\newblock Viscous flow past a body translating by time-periodic motion with
  zero average.
\newblock {\em To appear in Arch. Rat. Mech. Anal.\/} (2020).

\bibitem{GD03}
{\sc Granas, A., and Dugundji, J.}
\newblock {\em Fixed point theory}.
\newblock Springer Monogr. Math. New York, NY: Springer, 2003.

\bibitem{Gr08}
{\sc Grandmont, C.}
\newblock Existence of weak solutions for the unsteady interaction of a viscous
  fluid with an elastic plate.
\newblock {\em SIAM J. Math. Anal. 40}, 2 (2008), 716--737.

\bibitem{KamSchSpe23}
{\sc Kampschulte, M., Schwarzacher, S., and Sperone, G.}
\newblock Unrestricted deformations of thin elastic structures interacting with
  fluids.
\newblock {\em accepted in JMPA\/} (2023).

\bibitem{kreml2023time}
{\sc Kreml, O., M{\'a}cha, V., Ne{\v{c}}asov{\'a}, {\v{S}}., and
  Trifunovi{\'c}, S.}
\newblock On time-periodic solutions to an interaction problem between
  compressible viscous fluids and viscoelastic beams.
\newblock {\em arXiv preprint arXiv:2307.02687\/} (2023).

\bibitem{LR14}
{\sc Lengeler, D., and R\r{u}{\v{z}}i{\v{c}}ka, M.}
\newblock Weak solutions for an incompressible {Newtonian} fluid interacting
  with a {Koiter} type shell.
\newblock {\em Arch. Ration. Mech. Anal. 211}, 1 (2014), 205--255.

\bibitem{our-paper}
{\sc M{\^{\i}}ndril{\u{a}}, C., and Schwarzacher, S.}
\newblock Time-periodic weak solutions for an incompressible {Newtonian} fluid
  interacting with an elastic plate.
\newblock {\em SIAM J. Math. Anal. 54}, 4 (2022), 4139--4162.

\bibitem{muha-canic-arma-13}
{\sc Muha, B., and Cani{\'c}, S.}
\newblock Existence of a weak solution to a nonlinear fluid-structure
  interaction problem modeling the flow of an incompressible, viscous fluid in
  a cylinder with deformable walls.
\newblock {\em Arch. Ration. Mech. Anal. 207}, 3 (2013), 919--968.

\bibitem{muha-canic-noslip}
{\sc Muha, B., and {\v{C}}ani{\'c}, S.}
\newblock Existence of a weak solution to a fluid-elastic structure interaction
  problem with the {Navier} slip boundary condition.
\newblock {\em J. Differ. Equations 260}, 12 (2016), 8550--8589.

\bibitem{MS22}
{\sc Muha, B., and Schwarzacher, S.}
\newblock Existence and regularity of weak solutions for a fluid interacting
  with a non-linear shell in three dimensions.
\newblock {\em Annales de l'Institut Henri Poincar{\'e} C 39\/} (2023),
  1369--1412.

\bibitem{pironneau1994optimal}
{\sc Pironneau, O.}
\newblock Optimal shape design for aerodynamics.
\newblock {\em AGARD REPORT 803\/} (1994).

\bibitem{Quarteroni2000}
{\sc Quarteroni, A., Tuveri, M., and Veneziani, A.}
\newblock Computational vascular fluid dynamics: problems, models and methods.
\newblock {\em Computing and Visualization in Science 2\/} (2000), 163--197.
\newblock 10.1007/s007910050039.

\bibitem{turek.s.hron.j:numerical}
{\sc Turek, S., and Hron, J.}
\newblock Numerical techniques for multiphase flow with liquid-solid
  interaction.
\newblock In {\em Hemodynamical flows}, vol.~37 of {\em Oberwolfach Semin.}
  Birkh\"auser, Basel, 2008, pp.~379--501.

\bibitem{yang2020periodically}
{\sc Yang, R., Chong, K.~L., Wang, Q., Verzicco, R., Shishkina, O., and Lohse,
  D.}
\newblock Periodically modulated thermal convection.
\newblock {\em Physical review letters 125}, 15 (2020), 154502.

\end{thebibliography}

%
%
%\section{Appendix}\label{section:appendix}
%\subsection{Reynolds' transport theorem}

\end{document}